\documentclass{article}
\usepackage[T1]{fontenc}
\usepackage{color}
\usepackage{units}
\usepackage{mathtools}
\usepackage{amsmath}
\usepackage{amsthm}
\usepackage{amssymb}
\usepackage{stmaryrd}
\usepackage{esint}

\makeatletter
\theoremstyle{plain}
\newtheorem{thm}{\protect\theoremname}

\theoremstyle{definition}
\newtheorem{defn}[thm]{\protect\definitionname}
\theoremstyle{plain}
\newtheorem{prop}[thm]{\protect\propositionname}
\theoremstyle{plain}
\newtheorem{cor}[thm]{\protect\corollaryname}
\theoremstyle{plain}
\newtheorem{lem}[thm]{\protect\lemmaname}
\theoremstyle{plain}
\newtheorem{assumption}[thm]{\protect\assumptionname}
\theoremstyle{definition}\newtheorem*{rem*}{\protect\remarkname}\newtheorem{rem}[thm]{\protect\remarkname}
\newtheorem{example}[thm]{\protect\examplename}
\theoremstyle{plain}
\newtheorem{conjecture}[thm]{\protect\conjecturename}

\usepackage{amsmath}
\usepackage{paralist}
\usepackage{graphicx}
\usepackage[colorlinks=true]{hyperref}
\hypersetup{urlcolor=blue, citecolor=red}

\makeatother

\providecommand{\assumptionname}{Assumption}
\providecommand{\conjecturename}{Conjecture}
\providecommand{\corollaryname}{Corollary}
\providecommand{\definitionname}{Definition}
\providecommand{\examplename}{Example}
\providecommand{\lemmaname}{Lemma}
\providecommand{\propositionname}{Proposition}
\providecommand{\remarkname}{Remark}
\providecommand{\theoremname}{Theorem}
\title{Asymptotic stability of the multidimensional wave equation coupled with classes of positive-real impedance boundary conditions}

\author{Florian Monteghetti and Ghislain Haine and Denis Matignon}

\begin{document}

\maketitle
\centerline{\scshape Florian Monteghetti$^*$}
\medskip
{\footnotesize
\centerline{POEMS (CNRS-INRIA-ENSTA ParisTech), Palaiseau, France}
} 
\medskip
\centerline{\scshape Ghislain Haine and Denis Matignon}
\medskip
{\footnotesize
\centerline{ISAE-SUPAERO, Universit\'{e} de Toulouse, France}
}

\newcommand{\ABLV}{Arendt-Batty-Lyubich-V\~u}
\global\long\def\spaceR{\mathbb{R}}
\global\long\def\spaceC{\mathbb{C}}
\global\long\def\spaceN{\mathbb{N}}
\global\long\def\spaceZ{\mathbb{Z}}
\global\long\def\spaceContinuous{\mathcal{C}}
\global\long\def\spaceD{\mathcal{D}}
\global\long\def\spaceState{H}
\global\long\def\spaceDiff{V}
\global\long\def\spaceDomain{\mathcal{D}(\mathcal{A})}
\global\long\def\spaceBounded{\mathcal{L}}
\global\long\def\opA{\mathcal{A}}
\global\long\def\opK{\mathcal{K}}
\global\long\def\opId{\mathcal{I}}
\global\long\def\opT{\mathcal{T}}
\global\long\def\opZ{\mathcal{Z}}
\global\long\def\opDiv{\operatorname{div}}
\global\long\def\opKer{\operatorname{ker}}
\global\long\def\vector#1{\boldsymbol{#1}}
\global\long\def\idMat{I}
\global\long\def\normal{\boldsymbol{n}}
\global\long\def\transpose{\intercal}
\global\long\def\dinf{\text{d}}
\global\long\def\coordx{\boldsymbol{x}}
\global\long\def\Lyapunov{\Phi}
\global\long\def\hfun{z}
\global\long\def\zfun{z}
\global\long\def\state{X}
\global\long\def\pac{p}
\global\long\def\uac{\boldsymbol{u}}
\global\long\def\rat{\boldsymbol{\varphi}}
\global\long\def\diff{\varphi}
\global\long\def\delay{\chi}
\global\long\def\deriv{\eta}
\global\long\def\srcstate{F}
\global\long\def\srcpac{f_{p}}
\global\long\def\srcuac{\boldsymbol{f_{u}}}
\global\long\def\srcrat{\boldsymbol{f_{\varphi}}}
\global\long\def\srcdiff{f_{\varphi}}
\global\long\def\srcdelay{f_{\chi}}
\global\long\def\srcderiv{f_{\eta}}
\global\long\def\test{\psi}
\begin{abstract}  This paper proves the asymptotic stability of
the multidimensional wave equation posed on a bounded
open Lipschitz set, coupled with various classes of positive-real
impedance boundary conditions, chosen for their physical relevance:
time-delayed, standard diffusive (which includes the Riemann-Liouville
fractional integral) and extended diffusive (which includes the Caputo
fractional derivative). The method of proof consists in formulating
an abstract Cauchy problem on an extended state space using a dissipative
realization of the impedance operator, be it finite or infinite-dimensional.
The asymptotic stability of the corresponding strongly continuous
semigroup is then obtained by verifying the sufficient spectral conditions
derived by Arendt and Batty (Trans. Amer. Math. Soc., 306 (1988))
as well as Lyubich and V{\~u} (Studia Math., 88 (1988)).

\end{abstract}
\section{Introduction}

The broad focus
of this paper is the asymptotic stability of the wave equation with
so-called impedance boundary conditions (IBCs), also known as acoustic
boundary conditions.

Herein, the impedance operator, related to the Neumann-to-Dirichlet
map, is assumed to be continuous linear time-invariant, so that it
reduces to a time-domain convolution. \emph{Passive} convolution operators
\cite[\S~3.5]{beltrami2014distributions}, the kernels of which have
a positive-real Laplace transform, find applications in physics in
the modeling of locally-reacting energy absorbing material, such as
non perfect conductors in electromagnetism \cite{yuferev2010sibc}
and liners in acoustics \cite{monteghetti2017tdibc}. As a result,
IBCs are commonly used with Maxwell's equations \cite{hiptmair2014FastQuadratureIBC},
the linearized Euler equations \cite{monteghetti2017tdibc}, or the
wave equation \cite{sauter2017waveCQ}.

Two classes of convolution operators are well-known due to the ubiquity
of the physical phenomena they model. Slowly decaying kernels, which
yield so-called \emph{long-memory} operators, arise from losses without
propagation (due to e.g. viscosity or electrical/thermal resistance);
they include fractional kernels. On the other hand, lossless propagation,
encountered in acoustical cavity for instance, can be represented
as a \emph{time delay}. Both effects can be combined, so that time-delayed
long-memory operators model a propagation with losses.

Stabilization of the wave equation by a boundary damping, as opposed
to an internal damping, has been investigated in a wealth of works,
most of which employing the equivalent admittance formulation (\ref{eq:=00005BMOD=00005D_IBC-Admittance}),
see Remark~\ref{rem:=00005BMOD=00005D_Terminology} for the terminology.
Unless otherwise specified, the works quoted below deal with the multidimensional
wave equation.

Early studies established exponential stability with a proportional
admittance \cite{chen1981note,lagnese1983decay,komornik1990direct}.
A delay admittance is considered in \cite{nicaise2006stability},
where exponential stability is proven under a sufficient delay-independent
stability condition that can be interpreted as a passivity condition
of the admittance operator. The proof of well-posedness relies on
the formulation of an evolution problem using an infinite-dimensional
realization of the delay through a transport equation (see \cite[\S~VI.6]{engel2000semigroup}
\cite[\S~2.4]{curtainewart1995infinitedim} and references therein)
and stability is obtained using observability inequalities. The addition
of a $2$-dimensional realization to a delay admittance has been considered
in \cite{peralta2018delayibc}, where both exponential and asymptotic
stability results are shown under a passivity condition using the
energy multiplier method. See also \cite{wang2011stabilitydelay}
for a monodimensional wave equation with a non-passive delay admittance,
where it is shown that exponential stability can be achieved provided
that the delay is a multiple of the domain back-and-forth traveling
time.

A class of space-varying admittance with finite-dimensional realizations
have received thorough scrutiny in \cite{abbas2013polynomial} for
the monodimensional case and \cite{abbas2015stability} for the multidimensional
case. In particular, asymptotic stability is shown using the \ABLV{}
(ABLV) theorem in an extended state space.

Admittance kernels defined by a Borel measure on $(0,\infty)$ have
been considered in \cite{cornilleau2009ibc}, where exponential stability
is shown under an integrability condition on the measure \cite[Eq.~(7)]{cornilleau2009ibc}.
This result covers both distributed and discrete time delays, as well
as a class of integrable kernels. Other classes of integrable kernels
have been studied in \cite{desch2010stabilization,peralta2016delayibc,li2018memoryibc}.
Integrable kernels coupled with a $2$-dimensional realization are
considered in \cite{li2018memoryibc} using energy estimates. Kernels
that are both completely monotone and integrable are considered in
\cite{desch2010stabilization}, which uses the ABLV theorem on an
extended state space, and in \cite{peralta2016delayibc} with an added
time delay, which uses the energy method to prove exponential stability.
The energy multiplier method is also used in \cite{alabauboussouira2009ibc}
to prove exponential stability for a class of non-integrable singular
kernels.

The works quoted so far do not cover fractional kernels, which are
non-integrable, singular, and completely monotone. As shown in \cite{matignon2005asymptotic},
asymptotic stability results with fractional kernels can be obtained
with the ABLV theorem by using their realization; two works that follow
this methodology are \cite{matignon2014asymptotic}, which covers
the monodimensional Webster-Lokshin equation with a rational IBC,
and \cite{grabowski2013fracIBC}, which covers a monodimensional wave
equation with a fractional admittance.

The objective of this paper is to prove the asymptotic stability
of the multidimensional wave equation (\ref{eq:=00005BMOD=00005D_Wave-Equation})
coupled with a wide range of IBCs (\ref{eq:=00005BMOD=00005D_IBC})
chosen for their physical relevance. All the considered IBCs share
a common property: the Laplace transform of their kernel is a positive-real
function. A common method of proof, inspired by \cite{matignon2014asymptotic},
is employed that consists in formulating an abstract Cauchy problem
on an extended state space (\ref{eq:=00005BSTAB=00005D_Abstract-Cauchy-Problem})
using a realization of each impedance operator, be it finite or infinite-dimensional;
asymptotic stability is then obtained with the ABLV theorem, although
a less general alternative based on the invariance principle is also
discussed. In spite of the apparent unity of the approach, we are
not able to provide a single, unified proof: this leads us to formulate
a conjecture at the end of this work, which we hope will motivate
further works.

This paper is organized as follows. Section~\ref{sec:=00005BMOD=00005D_Model-and-preliminary-results}
introduces the model considered, recalls some known facts about positive-real
functions, formulates the ABLV theorem as Corollary~\ref{cor:=00005BSTAB=00005D_Asymptotic-Stability},
and establishes a preliminary well-posedness result in the Laplace
domain that is the cornerstone of the stability proofs. The
remaining sections demonstrate the applicability of Corollary~\ref{cor:=00005BSTAB=00005D_Asymptotic-Stability}
to IBCs with infinite-dimensional realizations that arise in physical
applications. Delay IBCs are covered in Section~\ref{sec:=00005BDELAY=00005D_Delay-impedance},
standard diffusive IBCs (e.g. fractional integral) are covered in
Section~\ref{sec:=00005BDIFF=00005D_Standard-diffusive-impedance},
while extended diffusive IBCs (e.g. fractional derivative) are covered
in Section~\ref{sec:=00005BEXTDIFF=00005D_Extended-diffusive-impedance}.
The extension of the obtained asymptotic stability results to IBCs
that contain a first-order derivative term is carried
out in Section~\ref{sec:=00005BDER=00005D_Addition-of-Derivative}.

\subsection*{Notation}

Vector-valued quantities are denoted in bold, e.g. $\vector f$. The
canonical scalar product in $\spaceC^{d}$, $d\in\llbracket1,\infty\llbracket$,
is denoted by $(\vector f,\vector g)_{\spaceC^{d}}\coloneqq\sum_{i=1}^{d}f_{i}\overline{g_{i}}$,
where $\overline{g_{i}}$ is the complex conjugate. Throughout the
paper, scalar products are antilinear with respect to the second argument.
Gradient and divergence are denoted by
\[
\nabla f\coloneqq\left[\partial_{i}f\right]_{i\in\llbracket1,d\rrbracket},\;\opDiv\vector f\coloneqq\sum_{i=1}^{d}\partial_{i}f_{i},
\]
where $\partial_{i}$ is the weak derivative with respect to the $i$-th
coordinate. The scalar product (resp. norm) on a Hilbert space $\spaceState$
is denoted by $(\cdot,\cdot)_{\spaceState}$ (resp. $\Vert\cdot\Vert_{\spaceState}$).
The only exception is the space of square integrable functions $(L^{2}(\Omega))^{d}$,
with $\Omega\subset\spaceR^{d}$ open set, for which the space is
omitted, i.e.
\[
(\vector f,\vector g)\coloneqq\int_{\Omega}(\vector f(\coordx),\vector g(\coordx))_{\spaceC^{d}}\,\dinf\coordx,\;\Vert\vector f\Vert\coloneqq\sqrt{(\vector f,\vector f)}.
\]
The scalar product on $(H^{1}(\Omega))^{d}$ is
\[
(\vector f,\vector g)_{H^{1}(\Omega)}\coloneqq(\vector f,\vector g)+(\nabla\vector f,\nabla\vector g).
\]
The topological dual of a Hilbert space $\spaceState$ is denoted
by $\spaceState^{'}$, and $L^{2}$ is used as a pivot space so that
for instance
\[
H^{\frac{1}{2}}\subset L^{2}\simeq(L^{2})^{'}\subset H^{-\frac{1}{2}},
\]
which leads to the following repeatedly used identity, for $\pac\in L^{2}$
and $\test\in H^{\frac{1}{2}}$,
\begin{equation}
\langle\pac,\test\rangle_{H^{-\frac{1}{2}},H^{\frac{1}{2}}}=\langle\pac,\test\rangle_{(L^{2})^{'},L^{2}}=(\pac,\overline{\test})_{L^{2}},\label{eq:=00005BPRE=00005D_L2-Pivot-Space}
\end{equation}
where $\langle\cdot,\cdot\rangle$ denotes the duality bracket (linear
in both arguments).
\begin{rem}
All the Hilbert spaces considered in this paper are over $\spaceC$.
\end{rem}
Other commonly used notations are $\spaceR^{*}\coloneqq\spaceR\backslash\{0\}$,
$\Re(s)$ (resp. $\Im(s)$) for the real (resp. imaginary) part of
$s\in\spaceC$, $A^{\transpose}$ for the transpose of a matrix $A$,
$R(A)$ (resp. $\opKer(A)$) for the range (resp. kernel) of $A$,
$\spaceContinuous(\Omega)$ for the space of continuous functions,
$\spaceContinuous_{0}^{\infty}(\Omega)$ for the space of infinitely
smooth and compactly supported functions, $\spaceD^{'}(\Omega)$ for
the space of distributions (dual of $\spaceContinuous_{0}^{\infty}(\Omega)$),
$\mathcal{E}^{'}(\Omega)$ for the space of compactly supported distributions,
$\spaceBounded(\spaceState)$ for the space of continuous linear operators
over $\spaceState$, $\overline{\Omega}$ for the closure of $\Omega$,
$Y_{1}:\,\spaceR\rightarrow\{0,1\}$ for the Heaviside function ($1$
over $(0,\infty)$, null elsewhere), and $\delta$ for the Dirac distribution.

\section{Model, strategy, and preliminary results\label{sec:=00005BMOD=00005D_Model-and-preliminary-results}}

Let $\Omega\subset\spaceR^{d}$ be a bounded open set. The Cauchy
problem considered in this paper is the wave equation under one of
its first-order form, namely
\begin{equation}
\partial_{t}\left(\begin{array}{c}
\uac\\
\pac
\end{array}\right)+\left(\begin{array}{c}
\nabla\pac\\
\opDiv\uac
\end{array}\right)=\vector 0\quad\text{on }\Omega,\label{eq:=00005BMOD=00005D_Wave-Equation}
\end{equation}
where $\uac(t,\coordx)\in\spaceC^{d}$ and $\pac(t,\coordx)\in\spaceC$.
To (\ref{eq:=00005BMOD=00005D_Wave-Equation}) is associated the so-called
\emph{impedance boundary condition} (IBC), formally defined as a time-domain
convolution between $\pac$ and $\uac\cdot\normal$,
\begin{equation}
\pac=z\star\uac\cdot\normal\quad\text{a.e. on }\partial\Omega,\label{eq:=00005BMOD=00005D_IBC}
\end{equation}
where $\normal$ is the unit outward normal and $z$ is the impedance
kernel. In general, $z$ is a causal distribution, i.e. $z\in\spaceD_{+}^{'}(\spaceR)$,
so that the convolution is to be understood in the sense of distributions
\cite[Chap.~III]{schwartz1966mathphys} \cite[Chap.~IV]{hormander1990pdevol1}.

This paper proves the asymptotic stability of strong solutions of
the evolution problem (\ref{eq:=00005BMOD=00005D_Wave-Equation},\ref{eq:=00005BMOD=00005D_IBC})
with an impedance kernel $z$ whose positive-real Laplace transform
is given by
\begin{equation}
\hat{z}(s)=\left(z_{0}+z_{\tau}e^{-\tau s}\right)+z_{1}s+\hat{z}_{\text{diff},1}(s)+s\,\hat{z}_{\text{diff},2}(s)\quad(\Re(s)>0),\label{eq:=00005BMOD=00005D_Target-Impedance}
\end{equation}
where $\tau>0$, $z_{\tau}\in\spaceR$, $z_{0}\geq\vert z_{\tau}\vert$,
$z_{1}>0$, and $z_{\text{diff},1}$ as well as $z_{\text{diff},2}$
are both locally integrable completely monotone kernels. The motivation
behind the definition of this kernel is physical as it models passive
systems that arise in e.g. electromagnetics \cite{garrappa2016models},
viscoelasticity \cite{desch1988exponential,mainardi1997frac}, and
acoustics \cite{helie2006diffusive,lombard2016fractional,monteghetti2016diffusive}.

By assumption, the right-hand side of (\ref{eq:=00005BMOD=00005D_Target-Impedance})
is a sum of positive-real kernels that each admit a dissipative realization.
This property enables to prove asymptotic stability with (\ref{eq:=00005BMOD=00005D_Target-Impedance})
by treating each of the four positive-real kernel separately: this
is carried out in Sections~\ref{sec:=00005BDELAY=00005D_Delay-impedance}--\ref{sec:=00005BDER=00005D_Addition-of-Derivative}.
This modularity property enables to keep concise notation by dealing
with the difficulty of each term one by one; it is illustrated in
Section~\ref{sec:=00005BDER=00005D_Addition-of-Derivative}. As already
mentioned in the introduction, the similarity between the four proofs
leads us to formulate a conjecture at the end of the paper.

The purpose of the remainder of this section is to present the strategy employed to establish asymptotic
stability as well as to prove preliminary results. Section~\ref{sub:=00005BMOD=00005D_Positive-real-facts}
justifies why, in order to obtain a well-posed problem
in $L^{2}$, the Laplace transform of the impedance kernel must be
a \emph{positive-real} function. Section~\ref{sub:=00005BSTAB=00005D_Strategy}
details the strategy used to establish asymptotic stability. Section~\ref{sub:=00005BMOD=00005D_Lemma-Rellich}
proves a consequence of the Rellich identity that is then used in
Section~\ref{sub:=00005BMOD=00005D_Well-posedness-Result} to obtain
a well-posedness result on the Laplace-transformed wave equation,
which will be used repeatedly.
\begin{rem}[Terminology]
\label{rem:=00005BMOD=00005D_Terminology} The boundary condition
(\ref{eq:=00005BMOD=00005D_IBC}) can equivalently be written as
\begin{equation}
\uac\cdot\normal=y\star\pac\quad\text{a.e. on }\partial\Omega,\label{eq:=00005BMOD=00005D_IBC-Admittance}
\end{equation}
where $y$ is known as the \emph{admittance} kernel ($y\star z=\delta$,
where $\delta$ is the Dirac distribution). This terminology can be
justified, for example, by the acoustical application: an acoustic
impedance is homogeneous to a pressure divided by a velocity. The
asymptotic stability results obtained in this paper still hold by
replacing the impedance by the admittance (in particular, the statement
``$z\neq0$'' becomes ``$y\neq0$''). The third way of formulating
(\ref{eq:=00005BMOD=00005D_IBC}), not considered in this paper, is
the so-called \emph{scattering} formulation \cite[p.~89]{beltrami2014distributions}
\cite[\S~2.8]{lozano2013dissipative}
\[
\pac-\uac\cdot\normal=\beta\star(\pac+\uac\cdot\normal)\quad\text{a.e. on }\partial\Omega,
\]
where $\beta$ is known as the \emph{reflection coefficient}. A Dirichlet
boundary condition is recovered for $z=0$ ($\beta=-\delta$) while
a Neumann boundary condition is recovered for $y=0$ ($\beta=+\delta$),
so that the proportional IBC, obtained for $z=z_{0}\delta$ ($\beta=\frac{z_{0}-1}{z_{0}+1}\,\delta$),
$z_{0}\geq0$, can be seen as an intermediate between the two.
\end{rem}

\begin{rem}
The use of a convolution in (\ref{eq:=00005BMOD=00005D_IBC}) can
be justified with the following classical result \cite[\S~III.3]{schwartz1966mathphys}
\cite[Thm.~1.18]{beltrami2014distributions}: if $\opZ$ is a linear
time-invariant and continuous mapping from $\mathcal{E}^{'}(\spaceR)$
into $\spaceD^{'}(\spaceR)$, then $\opZ(u)=\opZ(\delta)\star u$
for all $u\in\mathcal{E}^{'}(\spaceR)$.
\end{rem}

\subsection{Why positive-real kernels?\label{sub:=00005BMOD=00005D_Positive-real-facts}}

Assume that $(\uac,\pac)$ is a strong solution, i.e. that it belongs
to $\spaceContinuous([0,T];(H^{1}(\Omega))^{d+1})$. The elementary
a priori estimate
\begin{equation}
\Vert(\uac,\pac)(T)\Vert^{2}=\Vert(\uac,\pac)(0)\Vert^{2}-2\,\Re\left[\int_{0}^{T}(\pac(\tau),\uac(\tau)\cdot\normal)_{L^{2}(\partial\Omega)}\,\dinf\tau\right]\label{eq:=00005BMOD=00005D_Energy-Estimate}
\end{equation}
suggests that to obtain a contraction semigroup, the impedance kernel
must satisfy a passivity condition, well-known in system theory. This
justifies why we restrict ourselves to impedance kernels that are
\emph{admissible} in the sense of the next definition, adapted from
\cite[Def.~3.3]{beltrami2014distributions}.
\begin{defn}[Admissible impedance kernel]
\label{def:=00005BMOD=00005D_Admissible-Impedance} A distribution
$z\in\spaceD^{'}(\spaceR)$ is said to be an \emph{admissible impedance
kernel} if the operator $u\mapsto z\star u$ that maps $\mathcal{E}^{'}(\spaceR)$
into $\spaceD^{'}(\spaceR)$ enjoys the following properties: \textup{(i)}
causality, i.e. $z\in\spaceD_{+}^{'}(\spaceR)$; \textup{(ii)} reality,
i.e. real-valued inputs are mapped to real-valued outputs; \textup{(iii)}
passivity, i.e.
\begin{equation}
\forall u\in\spaceContinuous_{0}^{\infty}(\spaceR),\;\forall T>0,\;\Re\left[\int_{-\infty}^{T}(z\star u(\tau),u(\tau))_{\spaceC}\,\dinf\tau\right]\geq0.\label{eq:=00005BMOD=00005D_Passivity-Cond_Z}
\end{equation}

\end{defn}
An important feature of admissible impedance kernels $z$ is that
their Laplace transforms $\hat{z}$ are \emph{positive-real} functions,
see Definition~\ref{def:=00005BMOD=00005D_PR} and Proposition~\ref{prop:=00005BMOD=00005D_Characterization-LTI-Z}.
Herein, the Laplace transform $\hat{z}$ is an analytic function on
an open \emph{right} half-plane, i.e.
\[
\hat{z}(s)\coloneqq\int_{0}^{\infty}z(t)e^{-st}\,\dinf t\quad\left(s\in\spaceC_{c}^{+}\right),
\]
for some $c\geq0$ with
\[
\spaceC_{c}^{+}\coloneqq\{s\in\spaceC\,\vert\,\Re(s)>c\}.
\]
See \cite[Chap.~6]{schwartz1966mathphys} and \cite[Chap.~2]{beltrami2014distributions}
for the definition when $z\in\spaceD_{+}^{'}(\spaceR)$.
\begin{defn}[Positive-real function]
\label{def:=00005BMOD=00005D_PR}A function $f:\,\spaceC_{0}^{+}\rightarrow\spaceC$
is \emph{positive-real} if $f$ is analytic in $\spaceC_{0}^{+}$,
$f(s)\in\spaceR$ for $s\in(0,\infty)$, and $\Re[f(s)]\geq0$ for
$s\in\spaceC_{0}^{+}$.\end{defn}
\begin{prop}
\label{prop:=00005BMOD=00005D_Characterization-LTI-Z}A causal distribution
$z\in\spaceD_{+}^{'}(\spaceR)$ is an admissible impedance kernel
if and only if $\hat{z}$ is a positive-real function.\end{prop}
\begin{proof}
See \cite[\S~2.11]{lozano2013dissipative} for the case where the
kernel $z\in L^{1}(\spaceR)$ is a function and \cite[\S~3.5]{beltrami2014distributions}
for the general case where $z\in\spaceD_{+}^{'}(\spaceR)$ is a causal
distribution. (Note that, if $z$ is an admissible impedance kernel,
then $z$ is also tempered.)\end{proof}
\begin{rem}
\label{rem:=00005BMOD=00005D_PR-growth-at-infinity}The growth at
infinity of positive-real functions is at most polynomial. More specifically,
from the integral representation of positive-real functions \cite[Eq.~(3.21)]{beltrami2014distributions},
it follows that for $\Re(s)\geq c>0$, $\vert\hat{z}(s)\vert\leq C(c)P(\vert s\vert)$
where $P$ is a second degree polynomial.
\end{rem}

\subsection{Abstract framework for asymptotic stability\label{sub:=00005BSTAB=00005D_Strategy}}

Let the causal distribution $z\in\spaceD_{+}^{'}(\spaceR)$ be an
admissible impedance kernel. In order to prove the asymptotic stability
of (\ref{eq:=00005BMOD=00005D_Wave-Equation},\ref{eq:=00005BMOD=00005D_IBC}),
we will use the following strategy in Sections~\ref{sec:=00005BDELAY=00005D_Delay-impedance}--\ref{sec:=00005BDER=00005D_Addition-of-Derivative}.
We first rely on the knowledge of a realization of the impedance operator
$u\mapsto z\star u$ to formulate an abstract Cauchy problem on a
Hilbert space $\spaceState$,
\begin{equation}
\dot{\state}(t)=\opA\state,\;\state(0)=\state_{0}\in\spaceState,\label{eq:=00005BSTAB=00005D_Abstract-Cauchy-Problem}
\end{equation}
where the extended state $\state$ accounts for the memory of the
IBC. The scalar product $(\cdot,\cdot)_{\spaceState}$ is defined
using a Lyapunov functional associated with the realization. Since,
by design, the problem has the energy estimate $\Vert\state(t)\Vert_{\spaceState}\leq\Vert\state_{0}\Vert_{\spaceState}$,
it is natural to use the Lumer-Phillips theorem to show that the unbounded
operator
\begin{equation}
\opA:\,\spaceDomain\subset\spaceState\rightarrow\spaceState\label{eq:=00005BSTAB=00005D_Definition-A}
\end{equation}
generates a strongly continuous semigroup of contractions on $\spaceState$,
denoted by $\opT(t)$. For initial data in $\spaceDomain$, the function
\begin{equation}
t\mapsto\opT(t)\state_{0}\label{eq:=00005BSTAB=00005D_Solution}
\end{equation}
provides the unique strong solution in $\spaceContinuous([0,\infty);\spaceDomain)\cap\spaceContinuous^{1}([0,\infty);\spaceState)$
of the evolution problem (\ref{eq:=00005BSTAB=00005D_Abstract-Cauchy-Problem})
\cite[Thm.~1.3]{pazy1983stability}. For (less regular) initial data
in $\spaceState$, the solution is milder, namely $\spaceContinuous([0,\infty);\spaceState)$.

To prove the asymptotic stability of this solution, we rely upon the
following result, where we denote by $\sigma(\opA)$ (resp. $\sigma_{p}(\opA)$)
the spectrum (resp. point spectrum) of $\opA$ \cite[\S~VIII.1]{yosida1980funana}.
\begin{cor}
\label{cor:=00005BSTAB=00005D_Asymptotic-Stability}Let $\spaceState$
be a complex Hilbert space and $\opA$ be defined as (\ref{eq:=00005BSTAB=00005D_Definition-A}).
If
\begin{itemize}
\item[\textup{(i)}]  $\opA$ is dissipative, i.e. $\Re(\opA\state,\state)_{H}\leq0$
for every $\state\in\spaceDomain$,
\item[\textup{(ii)}]  $\opA$ is injective,
\item[\textup{(iii)}]  $s\opId-\opA$ is bijective for $s\in(0,\infty)\cup i\spaceR^{*}$,
\end{itemize}

then $\opA$ is the infinitesimal generator of a strongly continuous
semigroup of contractions $\opT(t)\in\spaceBounded(\spaceState)$
that is asymptotically stable, i.e.
\begin{equation}
\forall\state_{0}\in\spaceState,\;\Vert\opT(t)\state_{0}\Vert_{\spaceState}\underset{t\rightarrow\infty}{\longrightarrow}0.\label{eq:=00005BSTAB=00005D_T-asymptotic-stability}
\end{equation}

\end{cor}
\begin{proof}
The Lumer-Phillips theorem, recalled in Theorem~\ref{thm:=00005BSTAB=00005D_Lumer-Phillips},
shows that $\opA$ generates a strongly continuous semigroup of contractions
$\opT(t)\in\spaceBounded(\spaceState)$. In particular $\opA$ is
closed, from the Hille-Yosida theorem \cite[Thm.~3.1]{pazy1983stability},
so that the resolvent operator $(s\opId-\opA)^{-1}$ is closed whenever
it is defined. A direct application of the closed graph theorem \cite[\S~II.6]{yosida1980funana}
then yields
\[
\left\{ s\in\spaceC\;\vert\;s\opId-\opA\;\text{is bijective}\right\} \subset\rho(\opA),
\]
where $\rho(\opA)$ denotes the resolvent set of $\opA$ \cite[\S~VIII.1]{yosida1980funana}.
Hence $i\spaceR^{*}\subset\rho(\opA)$ and Theorem~\ref{thm:=00005BSTAB=00005D_Arendt-Batty}
applies since $0\notin\sigma_{p}(\opA)$.\end{proof}
\begin{rem}
Condition (iii) of Corollary~\ref{cor:=00005BSTAB=00005D_Asymptotic-Stability}
could be loosened by only requiring that $s\opId-\opA$ be surjective
for $s\in(0,\infty)$ and bijective for $s\in i\spaceR^{*}$. However,
in the proofs presented in this paper we always prove bijectivity
for $s\in(0,\infty)\cup i\spaceR^{*}$.
\end{rem}

\subsection{A consequence of the Rellich identity\label{sub:=00005BMOD=00005D_Lemma-Rellich}}

Using the Rellich identity, we prove below that the
Dirichlet and Neumann Laplacians do not have an eigenfunction in common.
\begin{prop}
\label{prop:=00005BMOD=00005D_Rellich-Lemma}Let $\Omega\subset\spaceR^{d}$
be a bounded open set. If $\pac\in H_{0}^{1}(\Omega)$ satisfies
\begin{equation}
\forall\test\in H^{1}(\Omega),\;(\nabla\pac,\nabla\test)=\lambda(\pac,\test)\label{eq:=00005BMOD=00005D_Rellich-WeakForm}
\end{equation}
for some $\lambda\in\spaceC$, then $\pac=0$ a.e.
in $\Omega$.\end{prop}
\begin{proof}
Let $\pac\in H_{0}^{1}(\Omega)$ be such that (\ref{eq:=00005BMOD=00005D_Rellich-WeakForm})
holds for some $\lambda\in\spaceC$. The proof is divided
in two steps.

(a) Let us first assume that $\partial\Omega$ is $\spaceContinuous^{\infty}$.
In particular,
\[
\forall\test\in H_{0}^{1}(\Omega),\;(\nabla\pac,\nabla\test)=\lambda(\pac,\test),
\]
so that $\pac$ is either null a.e. in $\Omega$ or an eigenfunction
of the Dirichlet Laplacian. In the latter case, since the boundary
$\partial\Omega$ is of class $\spaceContinuous^{\infty}$, we have
the regularity result $\pac\in\spaceContinuous^{\infty}(\overline{\Omega})$
\cite[Thm.~8.13]{gilbarg2001elliptic}. An integration by parts then
shows that, for $\test\in H^{1}(\Omega)$,
\[
(\partial_{n}\pac,\test)_{L^{2}(\partial\Omega)}=(\Delta\pac+\lambda\pac,\test)=0,
\]
so that $\partial_{n}\pac=0$ in $\partial\Omega$. However since
$\pac$ is $\spaceContinuous^{2}(\overline{\Omega})$ and $\partial\Omega$
is smooth we have \cite{rellich1940eigenfunctions}
\begin{equation}
\Vert\pac\Vert^{2}=\frac{\int_{\partial\Omega}(\partial_{n}\pac)^{2}\partial_{n}(\vert\coordx\vert^{2})\,\dinf\coordx}{4\lambda},\label{eq:=00005BMOD=00005D_Rellich-Formula}
\end{equation}
which shows that $\pac=0$ in $\Omega$. (The spectrum of the Dirichlet
Laplacian does not include $0$ \cite[\S~8.12]{gilbarg2001elliptic}.)

(b) Let us now assume that $\partial\Omega$ is not
$\spaceContinuous^{\infty}$. The strategy, suggested to us by Prof.
Patrick Ciarlet, is to get back to (a) by extending $\pac$ by zero.
Let $\mathcal{B}$ be an open ball such that $\overline{\Omega}\subset\mathcal{B}$.
We denote $\tilde{\pac}$ the extension of $\pac$ by zero, i.e. $\tilde{\pac}=\pac$
on $\Omega$ with $\tilde{\pac}$ null on $\mathcal{B}\backslash\Omega$.
From Proposition~\ref{prop:Extension-by-zero}, we have $\tilde{\pac}\in H_{0}^{1}\left(\mathcal{B}\right)$.
Using the definition of $\tilde{\pac}$, we can write
\begin{alignat*}{1}
\forall\test\in H^{1}(\mathcal{B}),\;(\nabla\tilde{\pac},\nabla\test)_{L^{2}\left(\mathcal{B}\right)} & =\left(\nabla\pac,\nabla\left[\test_{\vert\Omega}\right]\right)_{L^{2}\left(\Omega\right)}=\lambda\left(\pac,\test_{\vert\Omega}\right)_{L^{2}\left(\Omega\right)}=\lambda(\tilde{\pac},\test)_{L^{2}\left(\mathcal{B}\right)},
\end{alignat*}
so that applying (a) to $\tilde{\pac}\in H_{0}^{1}\left(\mathcal{B}\right)$
gives $\tilde{\pac}=0$ a.e. in $\mathcal{B}$.
\end{proof}

\subsection{A well-posedness result in the Laplace domain\label{sub:=00005BMOD=00005D_Well-posedness-Result}}

The following result will be used repeatedly. We define
\[
\overline{\spaceC_{0}^{+}}\coloneqq\{s\in\spaceC\,\vert\,\Re(s)\geq0\}.
\]

\begin{thm}
\label{thm:=00005BMOD=00005D_Weak-Form-p}Let $\Omega\subset\spaceR^{d}$
be a bounded open set with a Lipschitz boundary. Let
$z:\,\overline{\spaceC_{0}^{+}}\backslash\{0\}\rightarrow\spaceC_{0}^{+}$
be such that $z(s)\in\spaceR$ for $s\in(0,\infty)$. For every $s\in\overline{\spaceC_{0}^{+}}\backslash\{0\}$
and $l\in H^{-1}(\Omega)$ there exists a unique $\pac\in H^{1}(\Omega)$
such that
\begin{equation}
\forall\test\in H^{1}(\Omega),\;(\nabla\pac,\nabla\test)+s^{2}(\pac,\test)+\frac{s}{z(s)}(\pac,\test)_{L^{2}(\partial\Omega)}=\overline{l(\test)}.\label{eq:=00005BMOD=00005D_Weak-Form-p}
\end{equation}
Moreover, there is $C(s)>0$, independent of $\pac$, such that
\[
\Vert\pac\Vert_{H^{1}(\Omega)}\leq C(s)\,\Vert l\Vert_{H^{-1}(\Omega)}.
\]
\end{thm}
\begin{rem}
Note that $s\mapsto z(s)$ need not be continuous, so that Theorem~\ref{thm:=00005BMOD=00005D_Weak-Form-p}
can be used pointwise, i.e. for only some $s\in\overline{\spaceC_{0}^{+}}\backslash\{0\}$.
\end{rem}

\begin{rem}[Intuition]
 Although the need for Theorem~\ref{thm:=00005BMOD=00005D_Weak-Form-p}
will appear in the proofs of the next sections, let us give a \emph{formal}
motivation for the formulation (\ref{eq:=00005BMOD=00005D_Weak-Form-p}).
Assume that $(\uac,\pac)$ is a smooth solution of (\ref{eq:=00005BMOD=00005D_Wave-Equation},\ref{eq:=00005BMOD=00005D_IBC}).
Then $\pac$ solves the wave equation
\[
\partial_{t}^{2}\pac-\Delta\pac=0\quad\text{on }\Omega,
\]
with the impedance boundary condition
\[
\partial_{t}\pac=z\star\partial_{t}\uac\cdot\normal=z\star(-\nabla\pac)\cdot\normal=-z\star\partial_{n}\pac\quad\text{on }\partial\Omega,
\]
where $\partial_{n}\pac$ denotes the normal derivative of $\pac$
and the causal kernel $z$ is, say, tempered and locally integrable.
An integration by parts with $\test\in H^{1}(\Omega)$ reads
\[
(\nabla\pac,\nabla\psi)+(\partial_{t}^{2}\pac,\psi)-(\partial_{n}\pac,\psi)_{L^{2}(\partial\Omega)}=0.
\]
The formulation (\ref{eq:=00005BMOD=00005D_Weak-Form-p}) then follows
from the application of the Laplace transform in time, which gives
$\widehat{z\star\partial_{n}\pac}(s)=\hat{z}(s)\partial_{n}\hat{\pac}(s)$
and $\widehat{\partial_{t}\pac}(s)=s\hat{\pac}(s)$ assuming that
$\pac(t=0)=0$ on $\partial\Omega$.\end{rem}
\begin{proof}[Proof for $s\in(0,\infty)$.]
 If $s\in(0,\infty)$ this is an immediate consequence of the Lax-Milgram
lemma \cite[Thm.~6.6]{lax2002funana}. Define the following bilinear
form over $H^{1}(\Omega)\times H^{1}(\Omega)$:
\[
\overline{a(\pac,\test)}\coloneqq(\nabla\pac,\nabla\test)+s^{2}(\pac,\test)+\frac{s}{z(s)}(\pac,\test)_{L^{2}(\partial\Omega)}.
\]
Its boundedness follows from the continuity of the trace $H^{1}(\Omega)\rightarrow L^{2}(\partial\Omega)$
(see Section~\ref{sub:=00005BMISC=00005D_Embedding-Trace}). The
fact that $z(s)>0$ gives
\[
a(\test,\test) \geq\min(1,s^{2})\Vert\test\Vert_{H^{1}(\Omega)}^{2},
\]
which establishes the coercivity of $a$.
\end{proof}

\begin{proof}
Let $s\in\overline{\spaceC_{0}^{+}}\backslash\{0\}$. The Lax-Milgram
lemma does not apply since the sign of $\Re(\overline{s}z(s))$ is
indefinite in general, but the Fredholm alternative is applicable.
Using the Riesz-Fr\'echet representation theorem \cite[Thm.~6.4]{lax2002funana},
(\ref{eq:=00005BMOD=00005D_Weak-Form-p}) can be rewritten uniquely
as
\begin{equation}
(\opId-\opK(s))\pac=L\quad\text{in }H^{1}(\Omega),\label{eq:=00005BMOD=00005D_Weak-Form-p_Compact}
\end{equation}
where $L\in H^{1}(\Omega)$ satisfies $\overline{l(\test)}=(L,\test)_{H^{1}(\Omega)}$
and the operator $\opK(s)\in\spaceBounded(H^{1}(\Omega))$ is given
by
\[
(\opK(s)\pac,\text{\ensuremath{\test}})_{H^{1}(\Omega)}\coloneqq(1-s^{2})(\pac,\test)-\frac{s}{z(s)}(\pac,\test)_{L^{2}(\partial\Omega)}.
\]
The interest of (\ref{eq:=00005BMOD=00005D_Weak-Form-p_Compact})
lies in the fact that $\opK(s)$ turns out to be a compact operator,
see Lemma~\ref{lem:=00005BMOD=00005D_Weak-Form-p_K-compact}. The
Fredholm alternative states that $\opId-\opK(s)$ is injective if
and only if it is surjective \cite[Thm.~6.6]{brezis2011fun}. Using
Lemma~\ref{lem:=00005BMOD=00005D_Weak-Form-p_K-eigenvalue} and the
open mapping theorem \cite[\S~II.5]{yosida1980funana}, we conclude
that $\opId-\opK(s)$ is a bijection with continuous inverse, which
yields the claimed well-posedness result.
\end{proof}

\begin{lem}
\label{lem:=00005BMOD=00005D_Weak-Form-p_K-compact}Let $s\in\overline{\spaceC_{0}^{+}}\backslash\{0\}$.
The operator $\opK(s)\in\spaceBounded(H^{1}(\Omega))$ is compact.
\begin{proof}
Let $\pac,\test\in H^{1}(\Omega)$. The Cauchy\textendash Schwarz
inequality and the continuity of the trace $H^{1}(\Omega)\rightarrow L^{2}(\partial\Omega)$
yield the existence of a constant $C>0$ such that
\[
\left|(\opK(s)\pac,\text{\ensuremath{\test}})_{H^{1}(\Omega)}\right|\leq\left(\left|1-s^{2}\right|\Vert\pac\Vert+C\left|\frac{s}{z(s)}\right|\Vert\pac\Vert_{L^{2}(\partial\Omega)}\right)\Vert\test\Vert_{H^{1}(\Omega)},
\]
from which we deduce
\[
\Vert\opK(s)\pac\Vert_{H^{1}(\Omega)}\leq\left|1-s^{2}\right|\Vert\pac\Vert+C\left|\frac{s}{z(s)}\right|\Vert\pac\Vert_{L^{2}(\partial\Omega)}.
\]
Let $\epsilon\in(0,\frac{1}{2})$. The continuous embedding $H^{\frac{1}{2}+\epsilon}(\Omega)\subset L^{2}(\Omega)$
and the continuity of the trace $H^{\frac{1}{2}+\epsilon}(\Omega)\rightarrow L^{2}(\partial\Omega)$,
see Section~\ref{sub:=00005BMISC=00005D_Embedding-Trace}, yield
\[
\Vert\opK(s)\pac\Vert_{H^{1}(\Omega)}\leq\left(\left|1-s^{2}\right|+C^{'}\left|\frac{s}{z(s)}\right|\right)\Vert\pac\Vert_{H^{\epsilon+\frac{1}{2}}(\Omega)}.
\]
The compactness of the embedding $H^{1}(\Omega)\subset H^{\frac{1}{2}+\epsilon}(\Omega)$,
see Section~\ref{sub:=00005BMISC=00005D_Embedding-Trace}, enables
to conclude.
\end{proof}
\end{lem}

\begin{lem}
\label{lem:=00005BMOD=00005D_Weak-Form-p_K-eigenvalue}Let $s\in\overline{\spaceC_{0}^{+}}\backslash\{0\}$.
The operator $\opId-\opK(s)$ is injective.
\begin{proof}
Assume that $\opId-\opK(s)$ is not injective. Then there exists $\pac\in H^{1}(\Omega)\backslash\{0\}$
such that $\opK(s)\pac=\pac$, i.e. for any $\test\in H^{1}(\Omega)$,
\begin{equation}
(\nabla\pac,\nabla\text{\ensuremath{\test}})+s^{2}(\pac,\test)+\frac{s}{z(s)}(\pac,\test)_{L^{2}(\partial\Omega)}=0.\label{eq:=00005BIBC=00005D_Weak-Form-p_K-eigenvalue_1}
\end{equation}
In particular, for $\test=\pac$,
\begin{equation}
z(s)\Vert\nabla\pac\Vert^{2}+s^{2}z(s)\Vert\pac\Vert^{2}+s\Vert\pac\Vert_{L^{2}(\partial\Omega)}^{2}=0.\label{eq:=00005BIBC=00005D_Weak-Form-p_K-eigenvalue_2}
\end{equation}
To derive a contradiction, we distinguish between $s\in\spaceC_{0}^{+}$
and $s\in i\spaceR^{*}$.\\
($s\in\spaceC_{0}^{+}$) This is a direct consequence of Lemma~\ref{lem:=00005BMOD=00005D_Polynomial-Root}.\\
($s\in i\spaceR^{*}$) Let $s=i\omega$ with $\omega\in\spaceR^{*}$.
Then (\ref{eq:=00005BIBC=00005D_Weak-Form-p_K-eigenvalue_2}) reads
\[
\begin{cases}
\Re(z(i\omega))\left(\Vert\nabla\pac\Vert^{2}-\omega^{2}\Vert\pac\Vert^{2}\right)=0\\
\Im(z(i\omega))\left(\Vert\nabla\pac\Vert^{2}-\omega^{2}\Vert\pac\Vert^{2}\right)+\omega\Vert\pac\Vert_{L^{2}(\partial\Omega)}^{2}=0,
\end{cases}
\]
so that $\pac\in H_{0}^{1}(\Omega)$. Going back to the first identity
(\ref{eq:=00005BIBC=00005D_Weak-Form-p_K-eigenvalue_1}), we therefore
have
\[
\forall\test\in H^{1}(\Omega),\;(\nabla\pac,\nabla\text{\ensuremath{\test}})=\omega^{2}(\pac,\test).
\]
The contradiction then follows from Proposition~\ref{prop:=00005BMOD=00005D_Rellich-Lemma}.
\end{proof}
\end{lem}

\begin{lem}
\label{lem:=00005BMOD=00005D_Polynomial-Root}Let $(a_{0},a_{1},a_{2})\in[0,\infty)^{3}$
and $z\in\spaceC_{0}^{+}$. The polynomial $s\mapsto za_{2}\,s^{2}+a_{1}\,s+za_{0}$
has no roots in $\spaceC_{0}^{+}$.\end{lem}
\begin{proof}
The only case that needs investigating is $a_{i}>0$ for $i\in\llbracket0,2\rrbracket$.
Let us denote by $\sqrt{\cdot}$ the branch of the square root that
has a nonnegative real part, with a cut on $(-\infty,0]$ (i.e. $\sqrt{\cdot}$
is analytic over $\spaceC\backslash(-\infty,0]$). The roots are given
by
\[
s_{\pm}\coloneqq\frac{a_{1}\overline{z}}{2a_{2}\vert z\vert^{2}}\left(-1\pm\sqrt{1-\gamma z^{2}}\right)
\]
with $\gamma\coloneqq4\frac{a_{0}a_{2}}{a_{1}^{2}}>0$ so that
\[
\Re(s_{\pm})=\frac{a_{1}}{2a_{2}\vert z\vert^{2}}f_{\pm}(z)\quad\text{with}\quad f_{\pm}(z)\coloneqq\Re\left[\overline{z}\left(-1\pm\sqrt{1-\gamma z^{2}}\right)\right].
\]
The function $f_{\pm}$ is continuous on $\spaceC_{0}^{+}\backslash[\gamma^{-\nicefrac{1}{2}},\infty)$
(but not analytic) and vanishes only on $i\spaceR$ (if $f_{\pm}(z)=0$,
then there is $\omega\in\spaceR$ such that $2\omega z=i\left(\omega^{2}-\gamma\vert z\vert^{4}\right)$).
The claim therefore follows from
\[
f_{\pm}\left(\frac{1}{\sqrt{2\gamma}}\right)=\frac{-\sqrt{2}\pm1}{2\sqrt{\gamma}}<0.
\]

\end{proof}

In view of Theorem~\ref{thm:=00005BMOD=00005D_Weak-Form-p}, in the
remainder of this paper, we make the following assumption on the set
$\Omega$.
\begin{assumption}
The set $\Omega\subset\spaceR^{d}$, $d\in\llbracket1,\infty\llbracket$,
is a bounded open set with a Lipschitz boundary.
\end{assumption}

\section{Delay impedance\label{sec:=00005BDELAY=00005D_Delay-impedance}}

This section, as well as Sections~\ref{sec:=00005BDIFF=00005D_Standard-diffusive-impedance}
and \ref{sec:=00005BEXTDIFF=00005D_Extended-diffusive-impedance},
deals with IBCs that have an \emph{infinite}-dimensional realization,
which arise naturally in physical modeling \cite{monteghetti2016diffusive}.
Let us first consider the time-delayed impedance
\begin{equation}
\hat{z}(s)\coloneqq z_{0}+z_{\tau}e^{-\tau s},\label{eq:=00005BDELAY=00005D_Laplace}
\end{equation}
where $z_{0},z_{\tau},\tau\in\spaceR$, so that the corresponding
IBC (\ref{eq:=00005BMOD=00005D_IBC}) reads
\begin{equation}
\pac(t)=z_{0}\uac(t)\cdot\normal+z_{\tau}\uac(t-\tau)\cdot\normal\quad\text{a.e. on \ensuremath{\partial\Omega}},\;t>0.\label{eq:=00005BDELAY=00005D_IBC-Temp}
\end{equation}
The function (\ref{eq:=00005BDELAY=00005D_Laplace}) is positive-real
if and only if
\begin{equation}
z_{0}\geq\vert z_{\tau}\vert,\;\tau\geq0,\label{eq:=00005BDELAY=00005D_PR-Condition}
\end{equation}
which is assumed in the following. From now on, in addition to (\ref{eq:=00005BDELAY=00005D_PR-Condition}),
we further assume
\[
\hat{z}(0)\neq0,\;\tau\neq0.
\]
This section is organized as follows: a realization of $\hat{z}$
is recalled in Section~\ref{sub:=00005BDELAY=00005D_Time-delay-realization}
and the stability of the coupled system is shown in Section~\ref{sub:=00005BDELAY=00005D_Asymptotic-stability}.
\begin{rem}
In \cite{nicaise2006stability}, exponential (resp. asymptotic) stability
is shown under the condition $z_{0}>z_{\tau}>0$ (resp. $z_{0}\geq z_{\tau}>0$)
and $\tau>0$.
\end{rem}

\begin{rem}
The case of a (memoryless) proportional impedance $\hat{z}(s)\coloneqq z_{0}$
with $z_{0}>0$ is elementary (it is known that exponential stability
is achieved \cite{chen1981note,lagnese1983decay,komornik1990direct})
and can be covered by the strategy detailed in Section~\ref{sub:=00005BSTAB=00005D_Strategy}
without using an extended state space \cite[\S\,4.2.2]{monteghetti2018dissertation}.
\end{rem}

\subsection{Time-delay realization\label{sub:=00005BDELAY=00005D_Time-delay-realization}}

Following a well-known device, time-delays can be realized using a
transport equation on a bounded interval \cite[\S~2.4]{curtainewart1995infinitedim}
\cite[\S~VI.6]{engel2000semigroup}. Let $u$ be a causal input. The
linear time-invariant operator $u\mapsto z\star u$ can be realized
as
\[
z\star u(t)=z_{0}u(t)+z_{\tau}\delay(t,-\tau)\quad(t>0),
\]
where the state $\delay\in H^{1}(-\tau,0)$ with $t\geq0$ follows
the transport equation
\begin{equation}
\left\{ \begin{aligned} & \partial_{t}\delay(t,\theta)=\partial_{\theta}\delay(t,\theta), & \quad & \left(\theta\in(-\tau,0),\;t>0\right), &  & \text{(a)}\\
 & \delay(0,\theta)=0, &  & \left(\theta\in(-\tau,0)\right), &  & \text{(b)}\\
 & \delay(t,0)=u(t), &  & \left(t>0\right). &  & \text{(c)}
\end{aligned}
\right.\label{eq:=00005BDELAY=00005D_Transport-Equation}
\end{equation}
For $\delay\in\spaceContinuous^{1}([0,T];H^{1}(-\tau,0))$ solution
of (\ref{eq:=00005BDELAY=00005D_Transport-Equation}a), we have the
following energy balance
\begin{alignat*}{1}
\frac{1}{2}\frac{\dinf}{\dinf t}\Vert\delay(t,\cdot)\Vert_{L^{2}(-\tau,0)}^{2} & =\Re(\partial_{\theta}\delay(t,\cdot),\delay(t,\cdot))_{L^{2}(-\tau,0)}\\
 & =\frac{1}{2}\left[\vert\delay(t,0)\vert^{2}-\vert\delay(t,-\tau)\vert^{2}\right],
\end{alignat*}
which we shall use in the proof of Lemma~\ref{lem:=00005BDELAY=00005D_Dissipativity}.
\begin{rem}[Multiple delays]
Note that a finite number of time-delays $\tau_{i}>0$ can be accounted
for by setting $\tau\coloneqq\max_{i}\tau_{i}$ and writing
\[
z\star u(t)=z_{0}u(t)+\sum_{i}z_{\tau_{i}}\delay(t,-\tau_{i}).
\]
The corresponding impedance $\hat{z}(s)=z_{0}+\sum_{i}z_{\tau_{i}}e^{-\tau_{i}s}$
is positive-real if $z_{0}\geq\sum_{i}\vert z_{\tau_{i}}\vert$. No
substantial change to the proofs of Section~\ref{sub:=00005BDELAY=00005D_Asymptotic-stability}
is required to handle this case. In \cite{nicaise2006stability},
asymptotic stability is proven under the condition $z_{0}\geq\sum_{i}z_{i}>0$ and $z_i>0$.
\end{rem}

\subsection{Asymptotic stability\label{sub:=00005BDELAY=00005D_Asymptotic-stability}}

Let
\[
H_{\opDiv}(\Omega)\coloneqq\left\{ \uac\in L^{2}(\Omega)^{d}\;\vert\;\opDiv\uac\in L^{2}(\Omega)\right\} .
\]
The state space is defined as
\begin{equation}
\begin{gathered}\spaceState\coloneqq\nabla H^{1}(\Omega)\times L^{2}(\Omega)\times L^{2}(\partial\Omega;L^{2}(-\tau,0)),\\
((\uac,\pac,\delay),(\srcuac,\srcpac,\srcdelay))_{\spaceState}\coloneqq(\uac,\srcuac)+(\pac,\srcpac)+k(\delay,\srcdelay)_{L^{2}(\partial\Omega;L^{2}(-\tau,0))},
\end{gathered}
\label{eq:=00005BDELAY=00005D_Scalar-Prod}
\end{equation}
where $k\in\spaceR$ is a constant to be tuned to achieve dissipativity,
see Lemma~\ref{lem:=00005BDELAY=00005D_Dissipativity}. The evolution
operator is defined as
\begin{equation}
\begin{gathered}\spaceDomain\ni\state\coloneqq\left(\begin{array}{c}
\uac\\
\pac\\
\delay
\end{array}\right)\longmapsto\opA\state\coloneqq\left(\begin{array}{c}
-\nabla\pac\\
-\opDiv\uac\\
\partial_{\theta}\delay
\end{array}\right),\\
\spaceDomain\coloneqq\left\{ (\uac,\pac,\delay)\in\spaceState\;\left|\;\begin{alignedat}{1} & (\uac,\pac)\in H_{\opDiv}(\Omega)\times H^{1}(\Omega)\\
 & \delay\in L^{2}(\partial\Omega;H^{1}(-\tau,0))\\
 & \pac=z_{0}\uac\cdot\normal+z_{\tau}\delay(\cdot,-\tau)\;\text{in }H^{-\frac{1}{2}}(\partial\Omega)\\
 & \delay(\cdot,0)=\uac\cdot\normal\;\text{in }H^{-\frac{1}{2}}(\partial\Omega)
\end{alignedat}
\right.\right\} .
\end{gathered}
\label{eq:=00005BDELAY=00005D_Definition-A}
\end{equation}
In this formulation, the IBC (\ref{eq:=00005BDELAY=00005D_IBC-Temp})
is the third equation in $\spaceDomain$. We apply Corollary~\ref{cor:=00005BSTAB=00005D_Asymptotic-Stability},
see the Lemmas~\ref{lem:=00005BDELAY=00005D_Dissipativity}, \ref{lem:=00005BDELAY=00005D_Injectivity},
and \ref{lem:=00005BDELAY=00005D_Bijectivity} below. Lemma~\ref{lem:=00005BDELAY=00005D_Dissipativity}
shows that the seemingly free parameter $k$ must be restricted for
$\Vert\cdot\Vert_{\spaceState}$ to be a Lyapunov functional, as formally
highlighted in \cite{monteghetti2017delay}.
\begin{rem}[Bochner's integral]
\label{rem:=00005BDELAY=00005D_Bochner-Integration} For the integrability
of vector-valued functions, we follow the definitions and results
presented in \cite[\S~V.5]{yosida1980funana}. Let $\mathcal{B}$
be a Banach space. We have \cite[Thm.~V.5.1]{yosida1980funana}
\[
L^{2}(\partial\Omega;\mathcal{B})=\left\{ f:\,\partial\Omega\rightarrow\mathcal{B}\;\text{strongly measurable}\,\left|\,\Vert f\Vert_{\mathcal{B}}\in L^{2}(\partial\Omega)\right.\right\} .
\]
In Sections~\ref{sec:=00005BDIFF=00005D_Standard-diffusive-impedance}
and \ref{sec:=00005BEXTDIFF=00005D_Extended-diffusive-impedance},
we repeatedly use the following result: if $A\in\spaceBounded(\mathcal{B}_{1},\mathcal{B}_{2})$
and $u\in L^{2}(\partial\Omega;\mathcal{B}_{1})$, then $Au\in L^{2}(\partial\Omega;\mathcal{B}_{2})$.
\end{rem}

\begin{rem}
\label{rem:|Delay=00005D_Exclusion-Solenoidal-Fields}Since
$\nabla H^{1}(\Omega)$ is a closed subspace of $L^{2}(\Omega)^{d}$,
$\spaceState$ is a Hilbert space, see Section~\ref{sub:=00005BMISC=00005D_Hodge-decomposition}
for some background. In view of the orthogonal decomposition (\ref{eq:=00005BPRE=00005D_Hodge-Decomposition}),
working with $\nabla H^{1}\left(\Omega\right)$ instead of $L^{2}\left(\Omega\right)^{d}$
enables to get an injective evolution operator $\opA$. The exclusion
of the solenoidal fields $\uac$ that belong to $H_{\opDiv0,0}(\Omega)$
from the domain of $\opA$ can be physically justified by the fact
that these fields are non-propagating and are not affected by the
IBC.\end{rem}
\begin{lem}
\label{lem:=00005BDELAY=00005D_Dissipativity}The operator $\opA$
given by (\ref{eq:=00005BDELAY=00005D_Definition-A}) is dissipative
if and only if
\[
k\in\left[z_{0}-\sqrt{z_{0}^{2}-z_{\tau}^{2}},z_{0}+\sqrt{z_{0}^{2}-z_{\tau}^{2}}\right].
\]
\end{lem}
\begin{proof}
Let $\state\in\spaceDomain$. In particular, $\uac\cdot\normal\in L^{2}(\partial\Omega)$
since $\delay(\cdot,0)\in L^{2}(\partial\Omega)$. Using Green's formula
(\ref{eq:=00005BPRE=00005D_Green-Formula})
\begin{alignat*}{1}
\Re(\opA\state,\state)_{\spaceState}= & -\Re\left[\langle\uac\cdot\normal,\overline{\pac}\rangle_{H^{-\frac{1}{2}}(\partial\Omega),H^{\frac{1}{2}}(\partial\Omega)}\right]+k\,\Re\left(\partial_{\theta}\delay,\delay\right)_{L^{2}(\partial\Omega;L^{2}(-\tau,0))}\\
= & -\Re\left[(\uac\cdot\normal,\pac)_{L^{2}(\partial\Omega)}\right]+\frac{k}{2}\Vert\delay(\cdot,0)\Vert_{L^{2}(\partial\Omega)}^{2}-\frac{k}{2}\Vert\delay(\cdot,-\tau)\Vert_{L^{2}(\partial\Omega)}^{2},\\
= & \left(\frac{k}{2}-z_{0}\right)\Vert\delay(\cdot,0)\Vert_{L^{2}(\partial\Omega)}^{2}-\frac{k}{2}\Vert\delay(\cdot,-\tau)\Vert_{L^{2}(\partial\Omega)}^{2}\\
 & -z_{\tau}\Re\left[(\delay(\cdot,0),\delay(\cdot,-\tau))_{L^{2}(\partial\Omega)}\right],
\end{alignat*}
from which we deduce that $\opA$ is dissipative if and only if the
matrix
\[
\left[\begin{array}{cc}
z_{0}-\frac{k}{2} & \frac{z_{\tau}}{2}\\
\frac{z_{\tau}}{2} & \frac{k}{2}
\end{array}\right]
\]
is positive semidefinite, i.e. if and only if its determinant and
trace are nonnegative:
\[
(2z_{0}-k)k\geq z_{\tau}^{2}\quad\text{and}\quad z_{0}\geq0.
\]
The conclusion follows the expressions of the roots of $k\mapsto-k^{2}+2z_{0}k-z_{\tau}^{2}$.\end{proof}
\begin{lem}
\label{lem:=00005BDELAY=00005D_Injectivity}The operator $\opA$ given
by (\ref{eq:=00005BDELAY=00005D_Definition-A}) is injective. \end{lem}
\begin{proof}
Assume $\state\in\spaceDomain$ satisfies $\opA\state=0$, i.e. $\nabla\pac=\vector 0$,
$\opDiv\uac=0$, and
\begin{equation}
\partial_{\theta}\delay(\coordx,\theta)=0\quad\text{a.e. in }\partial\Omega\times(-\tau,0).\label{eq:=00005BDELAY=00005D_Injectivity-1}
\end{equation}
Hence $\delay(\coordx,\cdot)$ is constant with
\begin{equation}
\delay(\cdot,0)=\delay(\cdot,-\tau)=\uac\cdot\normal\quad\text{a.e. in }\mbox{\ensuremath{\partial\Omega}}.\label{eq:=00005BDELAY=00005D_Injectivity-2}
\end{equation}
Green's formula (\ref{eq:=00005BPRE=00005D_Green-Formula}) yields
\[
\langle\uac\cdot\normal,\overline{\pac}\rangle_{H^{-\frac{1}{2}}(\partial\Omega),H^{\frac{1}{2}}(\partial\Omega)}=0,
\]
and by combining with the IBC (i.e. the third equation
in $\spaceDomain$) and (\ref{eq:=00005BDELAY=00005D_Injectivity-2})
\[
\hat{z}(0)\Vert\uac\cdot\normal\Vert_{L^{2}(\partial\Omega)}^{2}=0,
\]
where we have used that $\uac\cdot\normal\in L^{2}(\partial\Omega)$
since $\delay(\cdot,0)\in L^{2}(\partial\Omega)$. Since $\hat{z}(0)\neq0$
we deduce that $\uac\in H_{\opDiv0,0}(\Omega)$, hence $\uac=\vector 0$
from (\ref{eq:=00005BPRE=00005D_Hodge-Decomposition}) and $\delay=0$.
The IBC gives $\pac=0$ a.e. on $\partial\Omega$, hence $\pac=0$
a.e. on $\Omega$.\end{proof}
\begin{lem}
\label{lem:=00005BDELAY=00005D_Bijectivity}Let $\opA$ be given by
(\ref{eq:=00005BDELAY=00005D_Definition-A}). Then, $s\opId-\opA$
is bijective for $s\in(0,\infty)\cup i\spaceR^{*}$.\end{lem}
\begin{proof}
Let $\srcstate\in\spaceState$ and $s\in(0,\infty)\cup i\spaceR^{*}$.
We seek a \emph{unique} $\state\in\spaceDomain$ such that $(s\opId-\opA)\state=\srcstate$,
i.e.
\begin{equation}
\begin{cases}
s\uac+\nabla\pac=\srcuac & \text{(a)}\\
s\pac+\opDiv\uac=\srcpac & \text{(b)}\\
s\delay-\partial_{\theta}\delay=\srcdelay. & \text{(c)}
\end{cases}\label{eq:=00005BDELAY=00005D_Bijectivity-1}
\end{equation}
The proof, as well as the similar ones found in the
next sections, proceeds in three steps.

(a) As a preliminary step, let us \emph{assume} that
(\ref{eq:=00005BDELAY=00005D_Bijectivity-1}) holds with $\state\in\spaceDomain$.
Equation (\ref{eq:=00005BDELAY=00005D_Bijectivity-1}c) can be uniquely
solved as
\begin{equation}
\delay(\cdot,\theta)=e^{s\theta}\uac\cdot\normal+R(s,\partial_{\theta})\srcdelay(\cdot,\theta),\label{eq:=00005BDELAY=00005D_Bijectivity-Delay}
\end{equation}
where we denote
\[
R(s,\partial_{\theta})\srcdelay(\coordx,\theta)\coloneqq\left[Y_{1}e^{s\cdot}\star\srcdelay(\coordx,\cdot)\right](\theta)=\int_{0}^{\theta}e^{s(\theta-\tilde{\theta})}\srcdelay(\coordx,\tilde{\theta})\,\dinf\tilde{\theta}.
\]
We emphasize that, in the remaining of the proof, $R(s,\partial_{\theta})$
is merely a convenient notation: the operator ``$\partial_{\theta}$''
cannot be defined independently from $\opA$ (see Remark~\ref{rem:=00005BDelay=00005D_Formal-Resolvent-Operator}
for a detailed explanation).

The IBC (i.e. the third equation in $\spaceDomain$)
can then be written as
\begin{equation}
\pac=\hat{z}(s)\uac\cdot\normal+z_{\tau}R(s,\partial_{\theta})\srcdelay(\cdot,-\tau)\quad\text{in }H^{-\frac{1}{2}}(\partial\Omega),\label{eq:=00005BDELAY=00005D_Bijectivity-IBC}
\end{equation}
and this identity actually takes place in $L^{2}(\partial\Omega)$
since
\[
\coordx\mapsto R(s,\partial_{\theta})\srcdelay(\coordx,-\tau)\in L^{2}(\partial\Omega).
\]
Let $\test\in H^{1}(\Omega)$. Combining $(\srcuac,\nabla\test)+s(\srcpac,\test)$
with (\ref{eq:=00005BDELAY=00005D_Bijectivity-IBC}) yields
\begin{equation}
\begin{alignedat}{1}(\nabla\pac,\nabla\test)+s^{2}(\pac,\test)+\frac{s}{\hat{z}(s)}(\pac,\test)_{L^{2}(\partial\Omega)}= & (\srcuac,\nabla\test)+s(\srcpac,\test)\\
 & +\frac{sz_{\tau}}{\hat{z}(s)}(R(s,\partial_{\theta})\srcdelay(\cdot,-\tau),\test)_{L^{2}(\partial\Omega)}.
\end{alignedat}
\label{eq:=00005BDELAY=00005D_Bijectivity-Weakp}
\end{equation}

In summary, $(s\opId-\opA)\state=\srcstate$ with $\state\in\spaceDomain$
implies (\ref{eq:=00005BDELAY=00005D_Bijectivity-Weakp}).

(b) We now construct a state $\state\in\spaceDomain$
such that $(s\opId-\opA)\state=\srcstate$. To do so, we use the conclusion
from the preliminary step (a).

Let $\pac\in H^{1}(\Omega)$ be the unique solution of (\ref{eq:=00005BDELAY=00005D_Bijectivity-Weakp})
obtained with Theorem~\ref{thm:=00005BMOD=00005D_Weak-Form-p}. It
remains to find suitable $\uac$ and $\delay$ so that $(\uac,\pac,\delay)\in\spaceDomain$.
Let us \emph{define} $\uac\in\nabla H^{1}(\Omega)$ by (\ref{eq:=00005BDELAY=00005D_Bijectivity-1}a).
Taking $\test\in\spaceContinuous_{0}^{\infty}(\Omega)$ in (\ref{eq:=00005BDELAY=00005D_Bijectivity-Weakp})
shows that $\uac\in H_{\opDiv}(\Omega)$ with (\ref{eq:=00005BDELAY=00005D_Bijectivity-1}b).
Using the expressions of $\uac$ and $\opDiv\uac$, and Green's formula
(\ref{eq:=00005BPRE=00005D_Green-Formula}), the weak formulation
(\ref{eq:=00005BDELAY=00005D_Bijectivity-Weakp}) can be rewritten
as
\[
\langle\uac\cdot\normal,\overline{\test}\rangle_{H^{-\frac{1}{2}}(\partial\Omega),H^{\frac{1}{2}}(\partial\Omega)}=\frac{1}{\hat{z}(s)}(\pac,\test)_{L^{2}(\partial\Omega)}-\frac{z_{\tau}}{\hat{z}(s)}(R(s,\partial_{\theta})\srcdelay(\cdot,-\tau),\test)_{L^{2}(\partial\Omega)}.
\]
which shows that $\pac$ and $\uac$ satisfy (\ref{eq:=00005BDELAY=00005D_Bijectivity-IBC}).
Let us now \emph{define} $\delay$ in $L^{2}(\partial\Omega;H^{1}(-\tau,0))$
by (\ref{eq:=00005BDELAY=00005D_Bijectivity-Delay}). By rewriting
(\ref{eq:=00005BDELAY=00005D_Bijectivity-IBC}) as
\[
\pac=(\hat{z}(s)-z_{\tau}e^{-s\tau})\uac\cdot\normal+z_{\tau}\left(e^{-s\tau}\uac\cdot\normal+R(s,\partial_{\theta})\srcdelay(\cdot,-\tau)\right)\quad\text{in }H^{-\frac{1}{2}}(\partial\Omega),
\]
we deduce thanks to (\ref{eq:=00005BDELAY=00005D_Laplace}) and (\ref{eq:=00005BDELAY=00005D_Bijectivity-Delay})
that the IBC holds, i.e. that $(\uac,\pac,\delay)\in\spaceDomain$.

(c) We now show the uniqueness in $\spaceDomain$ of
a solution of (\ref{eq:=00005BDELAY=00005D_Bijectivity-1}). The
uniqueness of $\pac$ in $H^{1}\left(\Omega\right)$
follows from Theorem~\ref{thm:=00005BMOD=00005D_Weak-Form-p}.
Although $\uac$ is not unique in $H_{\opDiv}(\Omega)$,
it is unique in $H_{\opDiv}(\Omega)\cap\nabla H^{1}\left(\Omega\right)$
following (\ref{eq:=00005BPRE=00005D_Hodge-Decomposition}). The uniqueness
of $\delay$ follows from the fact that (\ref{eq:=00005BDELAY=00005D_Bijectivity-1}c)
 is uniquely solvable in $\spaceDomain$.\end{proof}
\begin{rem}
\label{rem:=00005BDelay=00005D_Formal-Resolvent-Operator}In the proof,
$R(s,\partial_{\theta})$ is only a notation since $\partial_{\theta}$
(hence also its resolvent operator) cannot be defined separately from
$\opA$. Indeed, the definition of $\partial_{\theta}$ would be
\[
\left|\begin{alignedat}{1}\partial_{\theta}:\, & \spaceD(\partial_{\theta})\subset L^{2}(\partial\Omega;L^{2}(-\tau,0))\rightarrow L^{2}(\partial\Omega;L^{2}(-\tau,0))\\
 & \delay\mapsto\partial_{\theta}\delay,
\end{alignedat}
\right.
\]
with domain
\[
\spaceD(\partial_{\theta})=\left\{ \delay\in L^{2}(\partial\Omega;H^{1}(-\tau,0))\;\vert\;\delay(\cdot,0)=\uac\cdot\normal\right\}
\]
that depends upon $\uac$.
\end{rem}

\section{Standard diffusive impedance\label{sec:=00005BDIFF=00005D_Standard-diffusive-impedance}}

This section focuses on the class of so-called \emph{standard diffusive}
kernels \cite{montseny1998diffusive}, defined as
\begin{equation}
z(t)\coloneqq\int_{0}^{\infty}e^{-\xi t}Y_{1}(t)\,\dinf\mu(\xi),\label{eq:=00005BDIFF=00005D_Time-Standard-Diffusive}
\end{equation}
where $t\in\spaceR$ and $\mu$ is a positive Radon measure on $[0,\infty)$
that satisfies the following well-posedness condition
\begin{equation}
\int_{0}^{\infty}\frac{\dinf\mu(\xi)}{1+\xi}<\infty,\label{eq:=00005BDIFF=00005D_Well-posedness-Condition}
\end{equation}
which guarantees that $z\in L_{\textup{loc}}^{1}([0,\infty))$ with
Laplace transform
\begin{equation}
\hat{z}(s)=\int_{0}^{\infty}\frac{1}{s+\xi}\dinf\mu(\xi).\label{eq:=00005BDIFF=00005D_Laplace-Standard-Diffusive}
\end{equation}
The estimate
\begin{equation}
\forall s\in\overline{\spaceC_{0}^{+}}\backslash\{0\},\quad\frac{1}{\vert s+\xi\vert}\leq\sqrt{2}\max\left[1,\frac{1}{\vert s\vert}\right]\frac{1}{1+\xi},\label{eq:=00005BDIFF=00005D_FirstOrder-Estimate}
\end{equation}
which is used below, shows that $\hat{z}$ is defined on $\overline{\spaceC_{0}^{+}}\backslash\{0\}$.

This class of (positive-real) kernels is physically linked to non-propagating
lossy phenomena and arise in electromagnetics \cite{garrappa2016models},
viscoelasticity \cite{desch1988exponential,mainardi1997frac}, and
acoustics \cite{helie2006diffusive,lombard2016fractional,monteghetti2016diffusive}.
Formally, $\hat{z}$ admits the following realization
\begin{equation}
\left\{ \begin{alignedat}{1} & \vphantom{\int}\partial_{t}\diff(t,\xi)=-\xi\diff(t,\xi)+u(t),\;\diff(0,\xi)=0\quad\left(\xi\in(0,\infty)\right),\\
 & z\star u(t)=\int_{0}^{\infty}\diff(t,\xi)\,\dinf\mu(\xi).
\end{alignedat}
\right.\label{eq:=00005BDIFF=00005D_Diffusive-Realization}
\end{equation}
The realization (\ref{eq:=00005BDIFF=00005D_Diffusive-Realization})
can be given a meaning using the theory of well-posed linear systems
\cite{weiss2001wellposed,staffans2005well,matignon2010diffusivewp,tucsnak2014wellposed}.
 However, in order to prove asymptotic
stability, we need a framework to give a meaning to the \emph{coupled}
system (\ref{eq:=00005BMOD=00005D_Wave-Equation},\ref{eq:=00005BMOD=00005D_IBC},\ref{eq:=00005BDIFF=00005D_Diffusive-Realization}),
which, it turns out, can be done without defining a well-posed linear
system out of (\ref{eq:=00005BDIFF=00005D_Diffusive-Realization}).

Similarly to the previous sections, this section is divided into two
parts. Section~\ref{sub:=00005BDIFF=00005D_Abstract-realization}
defines the realization of (\ref{eq:=00005BDIFF=00005D_Diffusive-Realization})
and establishes some of its properties. These properties are then
used in Section~\ref{sub:=00005BDIFF=00005D_Asymptotic-stability}
to prove asymptotic stability of the coupled system.
\begin{rem}
The typical standard diffusive operator is the Riemann-Liouville fractional
integral \cite[\S~2.3]{samko1993fractional} \cite{matignon2008introduction}
\begin{equation}
\hat{z}(s)=\frac{1}{s^{\alpha}},\;\dinf\mu(\xi)=\frac{\sin(\alpha\pi)}{\pi}\frac{1}{\xi^{\alpha}}\dinf\xi,\label{eq:=00005BDIFF=00005D_Diffusive-Weight-Fractional}
\end{equation}
where $\alpha\in(0,1)$.
\end{rem}

\begin{rem}
\label{rem:=00005BDIFF=00005D_Completely-Monotonic}The expression
(\ref{eq:=00005BDIFF=00005D_Time-Standard-Diffusive}) arises naturally
when inverting multivalued Laplace transforms, see \cite[Chap.~4]{duffy2004transform}
for applications in partial differential equations. However, a standard
diffusive kernel can also be defined as follows: a causal kernel $z$
is said to be \emph{standard diffusive} if it belongs to $L_{\text{loc}}^{1}([0,\infty))$
and is completely monotone on $(0,\infty)$. By Bernstein's representation
theorem \cite[Thm.~5.2.5]{gripenberg1990volterra}, $z$ is standard
diffusive iff (\ref{eq:=00005BDIFF=00005D_Time-Standard-Diffusive},\ref{eq:=00005BDIFF=00005D_Well-posedness-Condition})
hold. Additionally, a standard diffusive kernel $z$ is integrable
on $(0,\infty)$ iff
\[
\mu(\{0\})=0\quad\text{and}\quad\int_{0}^{\infty}\frac{1}{\xi}\dinf\mu(\xi)<\infty,
\]
a property which will be referred to in Section~\ref{sub:=00005BDIFF=00005D_Abstract-realization}.
State spaces for the realization of classes of completely monotone
kernels have been studied in \cite{desch1988exponential,staffans1994well}.
\end{rem}

\subsection{Abstract realization\label{sub:=00005BDIFF=00005D_Abstract-realization}}

To give a meaning to (\ref{eq:=00005BDIFF=00005D_Diffusive-Realization})
suited for our purpose, we define, for any $s\in\spaceR$, the following
Hilbert space
\[
V_{s}\coloneqq\left\{ \diff:\,(0,\infty)\rightarrow\spaceC\text{ measurable}\;\left|\;\int_{0}^{\infty}\vert\diff(\xi)\vert^{2}(1+\xi)^{s}\,\dinf\mu(\xi)<\infty\right.\right\} ,
\]
with scalar product
\[
(\diff,\test)_{V_{s}}\coloneqq\int_{0}^{\infty}(\diff(\xi),\test(\xi))_{\spaceC}(1+\xi)^{s}\,\dinf\mu(\xi),
\]
so that the triplet $(\spaceDiff_{-1},\spaceDiff_{0},\spaceDiff_{1})$
satisfies the continuous embeddings
\begin{equation}
\spaceDiff_{1}\subset\spaceDiff_{0}\subset\spaceDiff_{-1}.\label{eq:=00005BDIFF=00005D_Triplet-Embedding}
\end{equation}
The space $\spaceDiff_{0}$ will be the energy space of the realization,
see (\ref{eq:=00005BDIFF=00005D_State-Space}). Note that the spaces
$\spaceDiff_{-1}$ and $\spaceDiff_{1}$ defined above are different
from those encountered when defining a well-posed linear system out
of (\ref{eq:=00005BDIFF=00005D_Diffusive-Realization}), see \cite{matignon2010diffusivewp}.
 When $\dinf\mu$ is given
by (\ref{eq:=00005BDIFF=00005D_Diffusive-Weight-Fractional}), the
spaces $\spaceDiff_{0}$ and $\spaceDiff_{1}$ reduce to the spaces
``$H_{\alpha}$'' and ``$V_{\alpha}$'' defined in \cite[\S~3.2]{matignon2014asymptotic}.

On these spaces, we wish to define the unbounded state operator $A$,
the control operator $B$, and the observation operator $C$ so that
\begin{equation}
A:\,\spaceD(A)\coloneqq\spaceDiff_{1}\subset\spaceDiff_{-1}\rightarrow\spaceDiff_{-1},\;B\in\spaceBounded(\spaceC,\spaceDiff_{-1}),\;C\in\spaceBounded(\spaceDiff_{1},\spaceC).\label{eq:=00005BDIFF=00005D_ABC-Definition}
\end{equation}
The state operator is defined as the following multiplication operator
\begin{equation}
\begin{alignedat}{1} & A:\\
\\
\end{alignedat}
\left|\begin{alignedat}{1} & \spaceD\left(A\right)\coloneqq\spaceDiff_{1}\subset\spaceDiff_{-1}\rightarrow\spaceDiff_{-1}\\
 & \diff\mapsto(\xi\mapsto-\xi\diff(\xi)).
\end{alignedat}
\right.\label{eq:=00005BDIFF=00005D_Multiplication-Operator-A}
\end{equation}
The control operator is simply
\begin{equation}
Bu\coloneqq\xi\mapsto u,\label{eq:=00005BDIFF=00005D_Application-Definition-B}
\end{equation}
and belongs to $\spaceBounded(\spaceC,\spaceDiff_{-1})$ thanks to
the condition (\ref{eq:=00005BDIFF=00005D_Well-posedness-Condition})
since, for $u\in\spaceC$,
\[
\Vert Bu\Vert_{\spaceDiff_{-1}}=\left[\int_{0}^{\infty}\frac{1}{1+\xi}\,\dinf\mu(\xi)\right]^{\nicefrac{1}{2}}\vert u\vert.
\]
The observation operator is
\[
C\diff\coloneqq\int_{0}^{\infty}\diff(\xi)\,\dinf\mu(\xi),
\]
and $C\in\spaceBounded(\spaceDiff_{1},\spaceC)$ thanks to (\ref{eq:=00005BDIFF=00005D_Well-posedness-Condition})
as, for $\diff\in\spaceDiff_{1}$,
\[
\vert C\diff\vert\leq\left[\int_{0}^{\infty}\frac{1}{1+\xi}\,\dinf\mu(\xi)\right]^{\nicefrac{1}{2}}\Vert\diff\Vert_{\spaceDiff_{1}}.
\]
The next lemma gathers properties of the triplet $(A,B,C)$ that are
used in Section~\ref{sub:=00005BDIFF=00005D_Asymptotic-stability}
to obtain asymptotic stability. Recall that if $A$ is closed and
$s\in\rho(A)$, then the resolvent operator $R(s,A)$ defined by (\ref{eq:=00005BSTAB=00005D_Resolvent-Operator})
belongs to $\spaceBounded(\spaceDiff_{-1},\spaceDiff_{1})$ \cite[\S~III.6.1]{kato1995ope}.
\begin{lem}
\label{lem:=00005BDIFF=00005D_Multiplication-Operator-A}The operator
$A$ defined by (\ref{eq:=00005BDIFF=00005D_Multiplication-Operator-A})
is injective, generates a strongly continuous semigroup of contractions
on $\spaceDiff_{-1}$, and satisfies $\overline{\spaceC_{0}^{+}}\backslash\{0\}\subset\rho(A)$. \end{lem}
\begin{proof}
The proof is split into three steps, (a), (b), and (c). (a) The injectivity
of $A$ follows directly from its definition. (b) Let us show that
$(0,\infty)\cup i\spaceR^{*}\subset\rho(A)$. Let $\srcdiff\in\spaceDiff_{-1}$,
$s\in(0,\infty)\cup i\spaceR^{*}$, and define
\begin{equation}
\diff(\xi)\coloneqq\frac{1}{s+\xi}\srcdiff(\xi)\quad\text{a.e. on }(0,\infty).\label{eq:=00005BDIFF=00005D_Multiplication-Operator-Resolvent}
\end{equation}
Using the estimate (\ref{eq:=00005BDIFF=00005D_FirstOrder-Estimate}),
we have
\begin{alignat*}{1}
\Vert\diff\Vert_{\spaceDiff_{1}} & \leq\sqrt{2}\max\left[1,\frac{1}{\vert s\vert}\right]\Vert\srcdiff\Vert_{\spaceDiff_{-1}},
\end{alignat*}
so that $\diff$ belongs to $\spaceDiff_{1}$ and $(s\opId-A)\diff=\srcdiff$
is well-posed. (c) For any $\diff\in\spaceDiff_{1}$, we have $\Re\left[(A\diff,\diff)_{\spaceDiff_{-1}}\right]\leq-\Vert\diff\Vert_{\spaceDiff_{0}}^{2}$,
so $A$ is dissipative. By the Lumer-Phillips theorem, $A$ generates
a strongly continuous semigroup of contractions on $\spaceDiff_{-1}$,
so that $\spaceC_{0}^{+}\subset\rho(A)$ \cite[Cor.~3.6]{pazy1983stability}.\end{proof}
\begin{lem}
\label{lem:=00005BDIFF=00005D_Admissible}The triplet of operators
$(A,B,C)$ defined above satisfies (\ref{eq:=00005BDIFF=00005D_ABC-Definition})
as well as the following properties.
\begin{enumerate}
\item[\textup{(i)}] (Stability) $A$ is closed and injective with $\overline{\spaceC_{0}^{+}}\backslash\{0\}\subset\rho(A)$.
\item[\textup{(ii)}] (Regularity)

\begin{enumerate}
\item $A\in\spaceBounded(\spaceDiff_{1},\spaceDiff_{-1})$.
\item For any $s\in\overline{\spaceC_{0}^{+}}\backslash\{0\}$,
\begin{equation}
  AR(s,A)_{\vert\spaceDiff_{0}}\in\spaceBounded(\spaceDiff_{0},\spaceDiff_{0}),
  \label{eq:=00005BDIFF=00005D_Definition-Regularity}
\end{equation}
where the vertical line denotes the restriction.
\end{enumerate}
\item[\textup{(iii)}] (Reality) For any $s\in(0,\infty)$,
\begin{equation}
CR(s,A)B_{\vert\spaceR}\in\spaceR,\label{eq:=00005BDIFF=00005D_Definition-Reality}
\end{equation}

\item[\textup{(iv)}] (Passivity) For any $(\diff,u)\in\spaceD(A\&B)$,
\begin{equation}
\Re\left[(A\diff+Bu,\diff)_{\spaceDiff_{0}}-(u,C\diff)_{\spaceC}\right]\leq0,\label{eq:=00005BDIFF=00005D_Definition-Dissipativity}
\end{equation}
where we define
\[
\spaceD(A\&B)\coloneqq\left\{ (\diff,u)\in\spaceDiff_{1}\times\spaceC\;\vert\;A\diff+Bu\in\spaceDiff_{0}\right\} .
\]

\end{enumerate}
\end{lem}
\begin{proof}
Let $A,$ $B$, and $C$ be defined as above. Each of the properties
is proven below.
\begin{itemize}
\item[(i)]  This condition is satisfied from Lemma~\ref{lem:=00005BDIFF=00005D_Multiplication-Operator-A}.
\item[(iia)]  Let $\diff\in\spaceDiff_{1}$. We have
\begin{alignat*}{1}
\Vert A\diff\Vert_{\spaceDiff_{-1}}^{2}= & \int_{0}^{\infty}\vert\diff(\xi)\vert^{2}\frac{\xi^{2}}{1+\xi}\,\dinf\mu(\xi)\\
\leq & \int_{0}^{\infty}\vert\diff(\xi)\vert^{2}(1+\xi)\,\dinf\mu(\xi)=\Vert\diff\Vert_{\spaceDiff_{1}}^{2},
\end{alignat*}
using the inequality $\xi^{2}\leq(1+\xi)^{2}.$
\item[(iib)]  Let $\srcdiff\in\spaceDiff_{0}$ and $s\in\overline{\spaceC_{0}^{+}}\backslash\{0\}$,
\[
\Vert AR(s,A)\srcdiff\Vert_{\spaceDiff_{0}}=\left[\int_{0}^{\infty}\left|\frac{\xi}{s+\xi}\srcdiff\right|^{2}\,\dinf\mu(\xi)\right]^{\nicefrac{1}{2}}\leq\Vert\srcdiff\Vert_{\spaceDiff_{0}},
\]
where we have used $\left|\frac{\xi}{s+\xi}\right|\leq\frac{\xi}{\Re(s)+\xi}\leq1.$
\item[(iii)]  Let $s\in(0,\infty)$ and $u\in\spaceR$. The reality condition
is fulfilled since
\[
CR(s,A)Bu=u\int_{0}^{\infty}\frac{\dinf\mu(\xi)}{s+\xi}.
\]

\item[(iv)]  Let $(\diff,u)\in\spaceD(A\&B)$. We have
\begin{equation}
\Re\left[(A\diff+Bu,\diff)_{\spaceDiff_{0}}-(u,C\diff)_{\spaceC}\right]=-\Re\left[\int_{0}^{\infty}\xi\vert\diff(\xi)\vert^{2}\,\dinf\mu(\xi)\right]\leq0,\label{eq:=00005BDIFF=00005D_Application-Passivity}
\end{equation}
so that the passivity condition is satisfied.
\end{itemize}
\end{proof}
\begin{rem}
The space $\spaceD(A\&B)$ is nonempty. Indeed, it contains at least
the following one dimensional subspace
\[
\left\{ (\diff,u)\in\spaceDiff_{1}\times\spaceC\;\vert\;\diff=R(s,A)Bu\right\}
\]
for any $s\in\rho(A)$ (which is nonempty from Lemma~\ref{lem:=00005BDIFF=00005D_Admissible}(i));
this follows from
\begin{alignat*}{1}
A\diff+Bu= & AR(s,A)Bu+Bu\\
= & sR(s,A)Bu\in\spaceDiff_{1}.
\end{alignat*}
It also contains $\left\{ (R(s,A)\diff,0)\;\vert\;\diff\in\spaceDiff_{0}\right\} $.
\end{rem}
For any $s\in\rho(A)$, we define
\begin{equation}
\hfun\coloneqq s\mapsto CR(s,A)B,\label{eq:=00005BDIFF=00005D_Transfer-Function-ABCD}
\end{equation}
which is analytic, from the analyticity of $R(\cdot,A)$ \cite[Thm.~III.6.7]{kato1995ope}.
Additionally, we have $\hfun(s)\in\spaceR$ for $s\in(0,\infty)$
from (\ref{eq:=00005BDIFF=00005D_Definition-Reality}), and $\Re(\hfun(s))\geq0$
from the passivity condition (\ref{eq:=00005BDIFF=00005D_Definition-Dissipativity})
with $\diff\coloneqq R(s,A)Bu\in\spaceD(A\&B)$:
\[
\Re(s)\Vert R(s,A)Bu\Vert_{\spaceDiff_{0}}^{2}\leq\Re\left[(u,\hfun(s)u)_{\spaceC}\right].
\]
Since $\spaceC_{0}^{+}\subset\rho(A)$, the function $\hfun$ defined
by (\ref{eq:=00005BDIFF=00005D_Transfer-Function-ABCD}) is positive-real.

\subsection{Asymptotic stability\label{sub:=00005BDIFF=00005D_Asymptotic-stability}}

Let $(A,B,C)$ be defined as in Section~\ref{sub:=00005BDIFF=00005D_Abstract-realization}.
We further assume that $A$, $B$, and $C$ are non-null operators.
The coupling between the wave equation (\ref{eq:=00005BMOD=00005D_Wave-Equation})
and the infinite-dimensional realization $(A,B,C)$ can be formulated
as the abstract Cauchy problem (\ref{eq:=00005BSTAB=00005D_Abstract-Cauchy-Problem})
using the following definitions. The extended state space is
\begin{equation}
\begin{gathered}\spaceState\coloneqq\nabla H^{1}(\Omega)\times L^{2}(\Omega)\times L^{2}(\partial\Omega;\spaceDiff_{0}),\\
((\uac,\pac,\diff),(\srcuac,\srcpac,\srcdiff))_{\spaceState}\coloneqq(\uac,\srcuac)+(\pac,\srcpac)+(\diff,\srcdiff)_{L^{2}(\partial\Omega;\spaceDiff_{0})},
\end{gathered}
\label{eq:=00005BDIFF=00005D_State-Space}
\end{equation}
and the evolution operator $\opA$ is
\begin{equation}
\begin{gathered}\spaceDomain\ni\state\coloneqq\left(\begin{array}{c}
\uac\\
\pac\\
\diff
\end{array}\right)\longmapsto\opA\state\coloneqq\left(\begin{array}{c}
-\nabla\pac\\
-\opDiv\uac\\
A\diff+B\uac\cdot\normal
\end{array}\right),\\
\spaceDomain\coloneqq\left\{ (\uac,\pac,\diff)\in\spaceState\;\left|\;\begin{alignedat}{1} & (\uac,\pac,\diff)\in H_{\opDiv}(\Omega)\times H^{1}(\Omega)\times L^{2}(\partial\Omega;\spaceDiff_{1})\\
 & (A\diff+B\uac\cdot\normal)\in L^{2}(\partial\Omega;\spaceDiff_{0})\\
 & \pac=C\diff\;\text{in }H^{\frac{1}{2}}(\partial\Omega)
\end{alignedat}
\right.\right\} ,
\end{gathered}
\label{eq:=00005BDIFF=00005D_Definition-A}
\end{equation}
where the IBC (\ref{eq:=00005BMOD=00005D_IBC},\ref{eq:=00005BDIFF=00005D_Diffusive-Realization})
is the third equation in $\spaceDomain$.
\begin{rem}
In the definition of $\opA$, there is an abuse of notation. Indeed,
we still denote by $A$ the following operator
\[
\left|\begin{alignedat}{1} & L^{2}(\partial\Omega;\spaceDiff_{1})\rightarrow L^{2}(\partial\Omega;\spaceDiff_{-1})\\
 & \diff\mapsto(\coordx\mapsto A\diff(\coordx,\cdot)),
\end{alignedat}
\right.
\]
which is well-defined from Lemma~\ref{lem:=00005BDIFF=00005D_Admissible}(iia)
and Remark~\ref{rem:=00005BDELAY=00005D_Bochner-Integration}. A
similar abuse of notation is employed for $B$ and $C$.
\end{rem}
Asymptotic stability is proven by applying Corollary~\ref{cor:=00005BSTAB=00005D_Asymptotic-Stability}
through Lemmas \ref{lem:=00005BDIFF=00005D_Dissipativity}, \ref{lem:=00005BDIFF=00005D_Injectivity},
and \ref{lem:=00005BDIFF=00005D_Bijectivity} below. In order to clarify
the proofs presented in Lemmas~\ref{lem:=00005BDIFF=00005D_Dissipativity}
and \ref{lem:=00005BDIFF=00005D_Injectivity}, we first prove a regularity
property on $\uac$ that follows from the definition of $\spaceDomain$.
\begin{lem}[Boundary regularity]
\label{lem:=00005BDIFF=00005D_Boundary-Regularity-u}If $\state=(\uac,\pac,\diff)\in\spaceDomain$,
then $\uac\cdot\normal\in L^{2}(\partial\Omega)$.\end{lem}
\begin{proof}
Let $\state\in\spaceDomain$. By definition of $\spaceDomain$, we
have $\diff\in L^{2}(\partial\Omega;\spaceDiff_{1})$ so that $A\diff\in L^{2}(\partial\Omega;\spaceDiff_{-1})$
from Lemma~\ref{lem:=00005BDIFF=00005D_Admissible}(iia) and Remark~\ref{rem:=00005BDELAY=00005D_Bochner-Integration}.
From
\[
B\uac\cdot\normal=\underbrace{A\diff+B\uac\cdot\normal}_{\mathclap{\in L^{2}(\partial\Omega;\spaceDiff_{0})}}-\overbrace{A\diff}^{\mathclap{\in L^{2}(\partial\Omega;\spaceDiff_{-1})}},
\]
we deduce that $B\uac\cdot\normal\in L^{2}(\partial\Omega;\spaceDiff_{-1})$.
The conclusion then follows from the definition of $B$ and the condition
(\ref{eq:=00005BDIFF=00005D_Well-posedness-Condition}).\end{proof}
\begin{lem}
\label{lem:=00005BDIFF=00005D_Dissipativity}The operator $\opA$
given by (\ref{eq:=00005BDIFF=00005D_Definition-A}) is dissipative.\end{lem}
\begin{proof}
Let $\state\in\spaceDomain$. In particular, $\uac\cdot\normal\in L^{2}(\partial\Omega)$
from Lemma~\ref{lem:=00005BDIFF=00005D_Boundary-Regularity-u}. Green's
formula (\ref{eq:=00005BPRE=00005D_Green-Formula}) and the inequality
(\ref{eq:=00005BDIFF=00005D_Definition-Dissipativity}) yield
\begin{alignat*}{1}
\Re(\opA\state,\state)_{\spaceState} & =\Re\left[(A\diff+B\uac\cdot\normal,\diff)_{L^{2}(\partial\Omega;\spaceDiff_{0})}-\langle\uac\cdot\normal,\overline{\pac}\rangle_{H^{-\frac{1}{2}}(\partial\Omega),H^{\frac{1}{2}}(\partial\Omega)}\right]\\
 & =\Re\left[(A\diff+B\uac\cdot\normal,\diff)_{L^{2}(\partial\Omega;\spaceDiff_{0})}-(\uac\cdot\normal,C\diff){}_{L^{2}(\partial\Omega)}\right]\leq0,
\end{alignat*}
where we have used that $\uac\cdot\normal\in L^{2}(\partial\Omega)$.\end{proof}
\begin{lem}
\label{lem:=00005BDIFF=00005D_Injectivity}The operator $\opA$ given
by (\ref{eq:=00005BDIFF=00005D_Definition-A}) is injective.\end{lem}
\begin{proof}
Assume $\state\in\spaceDomain$ satisfies $\opA\state=0$. In particular
$\nabla\pac=\vector 0$ and $\opDiv\uac=0$, so that Green's formula
(\ref{eq:=00005BPRE=00005D_Green-Formula}) yields
\[
\langle\uac\cdot\normal,\overline{\pac}\rangle_{H^{-\frac{1}{2}}(\partial\Omega),H^{\frac{1}{2}}(\partial\Omega)}=0,
\]
and by combining with the IBC (i.e. the third equation
in $\spaceDomain$)
\begin{equation}
(\uac\cdot\normal,C\diff)_{L^{2}(\partial\Omega)}=0,\label{eq:=00005BDIFF=00005D_Injectivity-1}
\end{equation}
where we have used that $\uac\cdot\normal\in L^{2}(\partial\Omega)$
from Lemma~\ref{lem:=00005BDIFF=00005D_Boundary-Regularity-u}. The
third equation that comes from $\opA\state=0$ is
\begin{equation}
A\diff(\coordx,\cdot)+B\uac(\coordx)\cdot\normal(\coordx)=0\quad\text{in }\spaceDiff_{0}\text{ for a.e. }\coordx\in\partial\Omega.\label{eq:=00005BDIFF=00005D_Injectivity-2}
\end{equation}
We now prove that $\state=0$, the key step being solving (\ref{eq:=00005BDIFF=00005D_Injectivity-2}).
Since $A$ is injective, (\ref{eq:=00005BDIFF=00005D_Injectivity-2})
has at most one solution $\diff\in L^{2}(\partial\Omega;\spaceDiff_{1})$.
Let us distinguish the possible cases.
\begin{itemize}
\item If $0\in\rho(A)$, then $\diff=R(0,A)B\uac\cdot\normal\in L^{2}(\partial\Omega;\spaceDiff_{1})$
is the unique solution. Inserting in (\ref{eq:=00005BDIFF=00005D_Injectivity-1})
and using (\ref{eq:=00005BDIFF=00005D_Transfer-Function-ABCD}) yields
\[
(\uac\cdot\normal,\zfun(0)\uac\cdot\normal)_{L^{2}(\partial\Omega)}=0,
\]
from which we deduce that $\uac\cdot\normal=0$ since $\zfun(0)$
is non-null.
\item If $0\in\sigma_{r}(A)\cup\sigma_{c}(A)$, then either $\overline{R(A)}\neq\spaceDiff_{-1}$
(definition of the residual spectrum) or $\overline{R(A)}=\spaceDiff_{-1}$
but $R(A)\neq\spaceDiff_{-1}$ (definition of the continuous spectrum
combined with the closed graph theorem, since $A$ is closed). $R\left(A\right)$
is equipped with the norm from $\spaceDiff_{-1}$. If $B\uac\cdot\normal\notin L^{2}(\partial\Omega;R(A))$,
then the only solution is $\diff=0$ and $\uac\cdot\normal=0$. If
$B\uac\cdot\normal\in L^{2}(\partial\Omega;R(A))$, then $\diff=-A^{-1}B\uac\cdot\normal$
is the unique solution, where $A^{-1}:\,R(A)\rightarrow\spaceDiff_{1}$
is an unbounded closed bijection. Inserting in (\ref{eq:=00005BDIFF=00005D_Injectivity-1})
yields
\[
(\uac\cdot\normal,(-CA^{-1}B)\uac\cdot\normal)_{L^{2}(\partial\Omega)}=0.
\]
Since $(-CA^{-1}B)\in\spaceC$ is non-null, we deduce that $\uac\cdot\normal=0$.
\end{itemize}

In summary, $\uac\in H_{\opDiv0,0}(\Omega)$, $\diff=0$ in $L^{2}(\partial\Omega;\spaceDiff_{1})$,
and $\pac=0$ in $L^{2}(\partial\Omega)$. The nullity of $\pac$
follows from $\nabla\pac=0$. The nullity of $\uac$ follows from
$H_{\opDiv0,0}(\Omega)\cap\nabla H^{1}(\Omega)=\{0\}$, see (\ref{eq:=00005BPRE=00005D_Hodge-Decomposition}).

\end{proof}
\begin{lem}
\label{lem:=00005BDIFF=00005D_Bijectivity}Let $\opA$ be given by
(\ref{eq:=00005BDIFF=00005D_Definition-A}). Then, $s\opId-\opA$
is bijective for $s\in(0,\infty)\cup i\spaceR^{*}$.\end{lem}
\begin{proof}
Let $\srcstate\in\spaceState$ and $s\in(0,\infty)\cup i\spaceR^{*}$.
We seek a \emph{unique} $\state\in\spaceDomain$ such that $(s\opId-\opA)\state=\srcstate$,
i.e.
\begin{equation}
\begin{cases}
s\uac+\nabla\pac=\srcuac & \text{(a)}\\
s\pac+\opDiv\uac=\srcpac & \text{(b)}\\
s\diff-A\diff-B\uac\cdot\normal=\srcdiff. & \text{(c)}
\end{cases}\label{eq:=00005BDIFF=00005D_Bijectivity-Eq}
\end{equation}
For later use, let us note that Equation~(\ref{eq:=00005BDIFF=00005D_Bijectivity-Eq}c)
and the IBC (i.e. the third equation in $\spaceDomain$)
imply
\begin{alignat}{2}
\diff & =R(s,A)(B\uac\cdot\normal+\srcdiff) &  & \quad\text{in }L^{2}(\partial\Omega;\spaceDiff_{1})\label{eq:=00005BDIFF=00005D_Bijectivity-DiffState}\\
\pac & =\zfun(s)\uac\cdot\normal+CR(s,A)\srcdiff &  & \quad\text{in }L^{2}(\partial\Omega).\label{eq:=00005BDIFF=00005D_Bijectivity-IBC}
\end{alignat}
Let $\test\in H^{1}(\Omega)$. Combining $(\srcuac,\nabla\test)+s(\srcpac,\test)$
with (\ref{eq:=00005BDIFF=00005D_Bijectivity-IBC}) yields
\begin{equation}
\begin{alignedat}{1}(\nabla\pac,\nabla\test)+s^{2}(\pac,\test)+\frac{s}{\zfun(s)}(\pac,\test)_{L^{2}(\partial\Omega)}=\; & (\srcuac,\nabla\test)+s(\srcpac,\test)\\
 & +\frac{s}{\zfun(s)}(CR(s,A)\srcdiff,\test)_{L^{2}(\partial\Omega)}.
\end{alignedat}
\label{eq:=00005BDIFF=00005D_Bijectivity-WeakForm-p}
\end{equation}
Note that since $CR(s,A)\in\spaceBounded(\spaceDiff_{-1},\spaceC),$
we have
\[
\coordx\mapsto CR(s,A)\srcdiff(\coordx)\in L^{2}(\partial\Omega),
\]
so that (\ref{eq:=00005BDIFF=00005D_Bijectivity-WeakForm-p}) is meaningful.
Moreover, we have $\Re(\zfun(s))\geq0$, and $\zfun(s)\in(0,\infty)$
for $s\in(0,\infty$). Therefore, we can apply Theorem~\ref{thm:=00005BMOD=00005D_Weak-Form-p},
pointwise, for $s\in(0,\infty)\cup i\spaceR^{*}$.

Let us denote by $\pac$ the unique solution of (\ref{eq:=00005BDIFF=00005D_Bijectivity-WeakForm-p})
in $H^{1}(\Omega)$, obtained from Theorem~\ref{thm:=00005BMOD=00005D_Weak-Form-p}.
It remains to find suitable $\uac$ and $\diff$.

Let us \emph{define} $\uac\in\nabla H^{1}(\Omega)$ by (\ref{eq:=00005BDIFF=00005D_Bijectivity-Eq}a).
Taking $\test\in\spaceContinuous_{0}^{\infty}(\Omega)$ in (\ref{eq:=00005BDIFF=00005D_Bijectivity-WeakForm-p})
shows that $\uac\in H_{\opDiv}(\Omega)$ and (\ref{eq:=00005BDIFF=00005D_Bijectivity-Eq}b)
holds. Using the expressions of $\uac$ and $\opDiv\uac$, and Green's
formula (\ref{eq:=00005BPRE=00005D_Green-Formula}), the weak formulation
(\ref{eq:=00005BDIFF=00005D_Bijectivity-WeakForm-p}) can be rewritten
as
\[
\langle\uac\cdot\normal,\overline{\test}\rangle_{H^{-\frac{1}{2}}(\partial\Omega),H^{\frac{1}{2}}(\partial\Omega)}=\zfun(s)^{-1}(\pac,\test)_{L^{2}(\partial\Omega)}-\zfun(s)^{-1}(CR(s,A)\srcdiff,\test)_{L^{2}(\partial\Omega)},
\]
which shows that $\pac$ and $\uac$ satisfy (\ref{eq:=00005BDIFF=00005D_Bijectivity-IBC}).

Let us now \emph{define} $\diff$ with (\ref{eq:=00005BDIFF=00005D_Bijectivity-DiffState}),
which belongs to $L^{2}(\partial\Omega;\spaceDiff_{1})$.
By rewriting (\ref{eq:=00005BDIFF=00005D_Bijectivity-IBC}) as
\[
\pac=(\zfun(s)-CR(s,A)B)\uac\cdot\normal+CR(s,A)(B\uac\cdot\normal+\srcdiff),
\]
we obtain from (\ref{eq:=00005BDIFF=00005D_Transfer-Function-ABCD})
and (\ref{eq:=00005BDIFF=00005D_Bijectivity-DiffState}) that the
IBC holds.

To obtain $(\uac,\pac,\diff)\in\spaceDomain$ it remains to show that
$A\diff+B\uac\cdot\normal$ belongs to $L^{2}(\partial\Omega;\spaceDiff_{0})$.
Using the definition (\ref{eq:=00005BDIFF=00005D_Bijectivity-DiffState})
of $\diff$, we have
\begin{alignat*}{1}
A\diff+B\uac\cdot\normal & =AR(s,A)(B\uac\cdot\normal+\srcdiff)+B\uac\cdot\normal\\
 & =(AR(s,A)+\opId)B\uac\cdot\normal+AR(s,A)\srcdiff\\
 & =sR(s,A)B\uac\cdot\normal+AR(s,A)\srcdiff.
\end{alignat*}
Since $\uac\cdot\normal\in L^{2}(\partial\Omega)$ and $R(s,A)B\in\spaceBounded(\spaceC,\spaceDiff_{1})$,
we have
\[
sR(s,A)B\uac\cdot\normal\in L^{2}(\partial\Omega;\spaceDiff_{1}).
\]
The property (\ref{eq:=00005BDIFF=00005D_Definition-Regularity})
implies that
\[
AR(s,A)\srcdiff\in L^{2}(\partial\Omega;\spaceDiff_{0}),
\]
hence that $(\uac,\pac,\diff)\in\spaceDomain$.

The uniqueness of $\pac$ follows from Theorem~\ref{thm:=00005BMOD=00005D_Weak-Form-p},
that of $\uac$ from (\ref{eq:=00005BPRE=00005D_Hodge-Decomposition}),
and that of $\diff$ from the bijectivity of $s\idMat-A$.
\end{proof}

\begin{rem}
The time-delay case does not fit into the framework proposed in Section~\ref{sub:=00005BDIFF=00005D_Abstract-realization},
see Remark~\ref{rem:=00005BDelay=00005D_Formal-Resolvent-Operator}.
This justifies why delay and standard diffusive IBCs are covered separately.
\end{rem}

\section{Extended diffusive impedance\label{sec:=00005BEXTDIFF=00005D_Extended-diffusive-impedance}}

In this section, we focus on a variant of the standard diffusive kernel,
namely the so-called \emph{extended diffusive} kernel given by
\begin{equation}
\hat{z}(s)\coloneqq\int_{0}^{\infty}\frac{s}{s+\xi}\dinf\mu(\xi),\label{eq:=00005BEXTDIFF=00005D_Laplace-Extended-Diffusive}
\end{equation}
where $\mu$ is a Radon measure that satisfies the condition (\ref{eq:=00005BDIFF=00005D_Well-posedness-Condition}),
already encountered in the standard case, and
\begin{equation}
\int_{0}^{\infty}\frac{1}{\xi}\dinf\mu(\xi)=\infty.\label{eq:=00005BEXTDIFF=00005D_Mu-Condition}
\end{equation}
The additional condition (\ref{eq:=00005BEXTDIFF=00005D_Mu-Condition})
implies that $t\mapsto\int_{0}^{\infty}e^{-\xi t}\,\dinf\mu(\xi)$
is not integrable on $(0,\infty)$, see Remark~\ref{rem:=00005BDIFF=00005D_Completely-Monotonic}.

From (\ref{eq:=00005BDIFF=00005D_Diffusive-Realization}), we directly
deduce that $\hat{z}$ \emph{formally} admits the realization
\begin{equation}
\left\{ \begin{alignedat}{1} & \vphantom{\int}\partial_{t}\diff(t,\xi)=-\xi\diff(t,\xi)+u(t),\;\diff(0,\xi)=0\quad\left(\xi\in(0,\infty)\right),\\
 & z\star u(t)=\int_{0}^{\infty}(-\xi\diff(t,\xi)+u(t))\,\dinf\mu(\xi),
\end{alignedat}
\right.\label{eq:=00005BEXTDIFF=00005D_Extended-Diffusive-Realization}
\end{equation}
where $u$ is a causal input. The separate treatment of the standard
(\ref{eq:=00005BDIFF=00005D_Laplace-Standard-Diffusive}) and extended
(\ref{eq:=00005BEXTDIFF=00005D_Laplace-Extended-Diffusive}) cases
is justified by the fact that physical models typically yield non-integrable
kernels, i.e.
\begin{equation}
\int_{0}^{\infty}\dinf\mu(\xi)=+\infty,\label{eq:=00005BEXTDIFF=00005D_Infinite-Integral}
\end{equation}
which prevents from splitting the observation integral in (\ref{eq:=00005BEXTDIFF=00005D_Extended-Diffusive-Realization}):
the observation and feedthrough operators must be
combined into $C\&D$. This justifies why (\ref{eq:=00005BEXTDIFF=00005D_Extended-Diffusive-Realization})
is only formal. Although a functional setting for (\ref{eq:=00005BEXTDIFF=00005D_Extended-Diffusive-Realization})
has been obtained in \cite[\S~B.3]{monteghetti2017delay}, we shall
again follow the philosophy laid out in Section~\ref{sec:=00005BDIFF=00005D_Standard-diffusive-impedance}.
Namely, Section~\ref{sub:=00005BEXTDIFF=00005D_Abstract-realization}
presents an abstract realization framework whose properties are given
in Lemma~\ref{lem:=00005BEXTDIFF=00005D_Admissible}, which slightly
differs from the standard case, and Section~\ref{sub:=00005BEXTDIFF=00005D_Asymptotic-stability}
shows asymptotic stability of the coupled system (\ref{eq:=00005BEXTDIFF=00005D_Definition-A}).
\begin{rem}
\label{rem:=00005BEXTDIFF=00005D_Fractional-Derivative} Let $\alpha\in(0,1)$.
The typical extended diffusive operator is the Riemman-Liouville fractional
derivative \cite[\S~2.3]{podlubny1999fractional} \cite{matignon2008introduction},
obtained for $\hat{z}(s)=s^{1-\alpha}$ and $\dinf\mu$ given by (\ref{eq:=00005BDIFF=00005D_Diffusive-Weight-Fractional}),
which satisfies the condition (\ref{eq:=00005BEXTDIFF=00005D_Mu-Condition}).
For this measure $\dinf\mu$, choosing the initialization $\varphi(0,\xi)=\nicefrac{u(0)}{\xi}$
in (\ref{eq:=00005BEXTDIFF=00005D_Extended-Diffusive-Realization})
yields the Caputo derivative \cite{lombard2016fractional}.
\end{rem}

\subsection{Abstract realization\label{sub:=00005BEXTDIFF=00005D_Abstract-realization}}

To give meaning to the realization (\ref{eq:=00005BEXTDIFF=00005D_Extended-Diffusive-Realization})
we follow a similar philosophy to the standard case, namely the definition
of a triplet of Hilbert spaces $(\spaceDiff_{-1},\spaceDiff_{0},\spaceDiff_{1})$
that satisfies the continuous embeddings (\ref{eq:=00005BDIFF=00005D_Triplet-Embedding})
as well as a suitable triplet of operators $(A,B,C)$.

The Hilbert spaces $\spaceDiff_{-1},$ $\spaceDiff_{0}$, and $\spaceDiff_{1}$
are defined as
\begin{alignat*}{1}
\spaceDiff_{1} & \coloneqq\left\{ \diff:\,(0,\infty)\rightarrow\spaceC\text{ measurable}\;\left|\;\int_{0}^{\infty}\vert\diff(\xi)\vert^{2}(1+\xi)\,\dinf\mu(\xi)<\infty\right.\right\} \\
\spaceDiff_{0} & \coloneqq\left\{ \diff:\,(0,\infty)\rightarrow\spaceC\text{ measurable}\;\left|\;\int_{0}^{\infty}\vert\diff(\xi)\vert^{2}\xi\,\dinf\mu(\xi)<\infty\right.\right\} \\
\spaceDiff_{-1} & \coloneqq\left\{ \diff:\,(0,\infty)\rightarrow\spaceC\text{ measurable}\;\left|\;\int_{0}^{\infty}\vert\diff(\xi)\vert^{2}\frac{\xi}{1+\xi^{2}}\,\dinf\mu(\xi)<\infty\right.\right\} ,
\end{alignat*}
with scalar products
\begin{alignat*}{1}
(\diff,\test)_{\spaceDiff_{1}} & \coloneqq\int_{0}^{\infty}(\diff(\xi),\test(\xi))_{\spaceC}(1+\xi)\,\dinf\mu(\xi)\\
(\diff,\test)_{\spaceDiff_{0}} & \coloneqq\int_{0}^{\infty}(\diff(\xi),\test(\xi))_{\spaceC}\,\xi\,\dinf\mu(\xi)\\
(\diff,\test)_{\spaceDiff_{-1}} & \coloneqq\int_{0}^{\infty}(\diff(\xi),\test(\xi))_{\spaceC}\frac{\xi}{1+\xi^{2}}\,\dinf\mu(\xi),
\end{alignat*}
so that the continuous embeddings (\ref{eq:=00005BDIFF=00005D_Triplet-Embedding})
are satisfied.  Note the change of definition of the energy space
$\spaceDiff_{0}$, which reflects the fact that the Lyapunov functional
of (\ref{eq:=00005BDIFF=00005D_Diffusive-Realization}) is different
from that of (\ref{eq:=00005BEXTDIFF=00005D_Extended-Diffusive-Realization}):
compare the energy balance (\ref{eq:=00005BDIFF=00005D_Application-Passivity})
with (\ref{eq:=00005BEXTDIFF=00005D_Application-Passivity}). The
change in the definition of $\spaceDiff_{-1}$ is a consequence of
this new definition of $\spaceDiff_{0}$. When $\dinf\mu$ is given
by (\ref{eq:=00005BDIFF=00005D_Diffusive-Weight-Fractional}), the
spaces $\spaceDiff_{0}$ and $\spaceDiff_{1}$ reduce to the spaces
``$\tilde{H}_{\alpha}$'' and ``$V_{\alpha}$'' defined in \cite[\S~3.2]{matignon2014asymptotic}.

The operators $A$, $B$, and $C$ satisfy (contrast with (\ref{eq:=00005BDIFF=00005D_ABC-Definition}))
\begin{equation}
A:\,\spaceD(A)\coloneqq\spaceDiff_{0}\subset\spaceDiff_{-1}\rightarrow\spaceDiff_{-1},\;B\in\spaceBounded(\spaceC,\spaceDiff_{-1}),\;C\in\spaceBounded(\spaceDiff_{1},\spaceC).\label{eq:=00005BEXTDIFF=00005D_ABC-Definition}
\end{equation}
 The state operator $A$ is still the multiplication operator (\ref{eq:=00005BDIFF=00005D_Multiplication-Operator-A}),
but with domain $\spaceDiff_{0}$ instead of $\spaceDiff_{1}$. Let
us check that this definition makes sense. For any $\diff\in\spaceDiff_{0}$,
we have
\begin{alignat}{1}
\Vert A\diff\Vert_{\spaceDiff_{-1}} & =\left[\int_{0}^{\infty}\vert\diff(\xi)\vert^{2}\frac{\xi^{3}}{1+\xi^{2}}\,\dinf\mu(\xi)\right]^{\nicefrac{1}{2}}\leq\Vert\diff\Vert_{\spaceDiff_{0}}.\label{eq:=00005BEXTDIFF=00005D_Estimate-Aphi}
\end{alignat}
The control operator $B$ is defined as (\ref{eq:=00005BDIFF=00005D_Application-Definition-B})
and we have for any $u\in\spaceC$
\begin{alignat*}{1}
\Vert Bu\Vert_{\spaceDiff_{-1}} & =\left[\int_{0}^{\infty}\vert u\vert^{2}\frac{\xi}{1+\xi^{2}}\,\dinf\mu(\xi)\right]^{\nicefrac{1}{2}}\leq\tilde{C}\left[\int_{0}^{\infty}\frac{1}{1+\xi}\,\dinf\mu(\xi)\right]^{\nicefrac{1}{2}}\vert u\vert,
\end{alignat*}
where the constant $\tilde{C}>0$ is
\[
\tilde{C}\coloneqq\left\Vert \frac{\xi(1+\xi)}{1+\xi^{2}}\right\Vert _{L^{\infty}(0,\infty)}.
\]
The observation operator $C$ is identical to the standard case. For
use in Section~\ref{sub:=00005BEXTDIFF=00005D_Asymptotic-stability},
properties of $(A,B,C)$ are gathered in Lemma~\ref{lem:=00005BEXTDIFF=00005D_Admissible}
below.
\begin{lem}
\label{lem:=00005BEXTDIFF=00005D_Multiplication-Operator-1}The operator
$A$ generates a strongly continuous semigroup of contractions on
$\spaceDiff_{-1}$ and satisfies  $\overline{\spaceC_{0}^{+}}\backslash\{0\}\subset\rho(A)$. \end{lem}
\begin{proof}
The proof is similar to that of Lemma~\ref{lem:=00005BDIFF=00005D_Multiplication-Operator-A}.
Let $s\in\overline{\spaceC_{0}^{+}}\backslash\{0\}$ and $\srcdiff\in\spaceDiff_{-1}$.
Let us define $\diff$ by (\ref{eq:=00005BDIFF=00005D_Multiplication-Operator-Resolvent}).
(a) We have
\begin{alignat*}{1}
\Vert\diff\Vert_{\spaceDiff_{0}} & =\left[\int_{0}^{\infty}\left|\frac{1}{s+\xi}\srcdiff\right|^{2}\xi\,\dinf\mu(\xi)\right]^{\nicefrac{1}{2}}\\
 & \leq\sqrt{2}\max\left[1,\frac{1}{\vert s\vert}\right]\left\Vert \frac{1+\xi^{2}}{(1+\xi)^{2}}\right\Vert _{L^{\infty}(0,\infty)}\Vert\srcdiff\Vert_{\spaceDiff_{-1}},
\end{alignat*}
so that $\diff$ solves $(s\opId-A)\diff=\srcdiff$ in $\spaceDiff_{0}$.
Since $s\opId-A$ is injective, we deduce that $s\in\rho(A)$. (b)
Let $\diff\in\spaceDiff_{0}$. We have
\[
(A\diff,\diff)_{\spaceDiff_{-1}}=-\int_{0}^{\infty}\vert\diff(\xi)\vert^{2}\frac{\xi^{2}}{1+\xi^{2}}\,\dinf\mu(\xi)\leq-\Vert\diff\Vert_{\spaceDiff_{0}}^{2},
\]
so that $A$ is dissipative. The conclusion follows from the Lumer-Phillips
theorem.
\end{proof}

\begin{lem}
\label{lem:=00005BEXTDIFF=00005D_Multiplication-Operator-2}The operators
$A$ and $B$ are injective. Moreover, if (\ref{eq:=00005BEXTDIFF=00005D_Mu-Condition})
holds, then $R(A)\cap R(B)=\{0\}$.\end{lem}
\begin{proof}
The injectivity of $A$ and $B$ is immediate. Let $\srcdiff\in R(A)\cap R(B)$,
so that there is $\diff\in\spaceDiff_{0}$ and $u\in\spaceC$ such
that $A\diff=Bu$, i.e. $-\xi\diff(\xi)=u$ a.e. on $(0,\infty)$.
The function $\diff$ belongs to $\spaceDiff_{0}$ if and only if
\[
\vert u\vert^{2}\int_{0}^{\infty}\frac{1}{\xi}\dinf\mu(\xi)<\infty.
\]
So that, assuming (\ref{eq:=00005BEXTDIFF=00005D_Mu-Condition}),
$\diff$ belongs to $\spaceDiff_{0}$ if and only if $u=0$ a.e on
$(0,\infty)$.\end{proof}
\begin{lem}
\label{lem:=00005BEXTDIFF=00005D_Admissible}The triplet of operators
$(A,B,C)$ defined above satisfies (\ref{eq:=00005BEXTDIFF=00005D_ABC-Definition})
as well as the following properties.
\begin{enumerate}
\item[\textup{(i)}] (Stability) $A$ is closed with $\overline{\spaceC_{0}^{+}}\backslash\{0\}\subset\rho(A)$
and satisfies
\begin{equation}
\forall(\diff,u)\in\spaceD(C\&D),\;A\diff=Bu\Rightarrow(\diff,u)=(0,0),\label{eq:=00005BEXTDIFF=00005D_Definition-Injectivity}
\end{equation}
where we define
\[
\spaceD(C\&D)\coloneqq\left\{ (\diff,u)\in\spaceDiff_{0}\times\spaceC\;\vert\;A\diff+Bu\in\spaceDiff_{1}\right\} .
\]

\item[\textup{(ii)}] (Regularity)

\begin{enumerate}
\item[\textup{(a)}] $A\in\spaceBounded(\spaceDiff_{0},\spaceDiff_{-1})$.
\item[\textup{(b)}] For any $s\in\overline{\spaceC_{0}^{+}}\backslash\{0\}$,
\begin{equation}
AR(s,A)_{\vert\spaceDiff_{0}}\in\spaceBounded(\spaceDiff_{0},\spaceDiff_{1}),\;R(s,A)B\in\spaceBounded(\spaceC,\spaceDiff_{1}).\label{eq:=00005BEXTDIFF=00005D_Definition-Regularity}
\end{equation}

\end{enumerate}
\item[\textup{(iii)}] (Reality) Identical to Lemma~\ref{lem:=00005BDIFF=00005D_Admissible}(iii).
\item[\textup{(iv)}] (Passivity) For any $(\diff,u)\in\spaceD(C\&D)$,
\begin{equation}
\Re\left[(A\diff+Bu,\diff)_{\spaceDiff_{0}}-(u,C(A\diff+Bu))_{\spaceC}\right]\leq0.\label{eq:=00005BEXTDIFF=00005D_Definition-Dissipativity}
\end{equation}

\end{enumerate}
\end{lem}
\begin{proof}
Let $(A,B,C)$ be as defined above. Each of the properties is proven
below.
\begin{enumerate}
\item[(i)]  Follows from Lemmas~\ref{lem:=00005BEXTDIFF=00005D_Multiplication-Operator-1}
and \ref{lem:=00005BEXTDIFF=00005D_Multiplication-Operator-2}.
\item[(iia)]  Follows from (\ref{eq:=00005BEXTDIFF=00005D_Estimate-Aphi}).
\item[(iib)]  Let $s\in\overline{\spaceC_{0}^{+}}\backslash\{0\}$, $\srcdiff\in\spaceDiff_{0}$,
and $u\in\spaceC$. We have
\begin{alignat*}{1}
\Vert AR(s,A)\srcdiff\Vert_{\spaceDiff_{1}} & =\left[\int_{0}^{\infty}\vert\srcdiff(\xi)\vert^{2}\frac{\xi^{2}(1+\xi)}{\vert s+\xi\vert^{2}}\,\dinf\mu(\xi)\right]^{\nicefrac{1}{2}}\\
 & \leq\sqrt{2}\max\left[1,\frac{1}{\vert s\vert}\right]\Vert\srcdiff\Vert_{\spaceDiff_{0}},
\end{alignat*}
and
\begin{alignat*}{1}
\Vert R(s,A)Bu\Vert_{\spaceDiff_{1}} & =\left(\int_{0}^{\infty}\frac{1+\xi}{\vert s+\xi\vert^{2}}\,\dinf\mu(\xi)\right)^{\nicefrac{1}{2}}\vert u\vert\\
 & \leq\sqrt{2}\max\left[1,\frac{1}{\vert s\vert}\right]\left(\int_{0}^{\infty}\frac{1}{1+\xi}\,\dinf\mu(\xi)\right)^{\nicefrac{1}{2}}\vert u\vert.
\end{alignat*}

\item[(iii)]  Immediate.
\item[(iv)]  Let $(\diff,u)\in\spaceD(C\&D)$. We have
\begin{alignat}{1}
\Re\Bigl[(A\diff+ & Bu,\diff)_{\spaceDiff_{0}}-(u,C(A\diff+Bu))_{\spaceC}\Bigr]\nonumber \\
 & =\Re\left[\int_{0}^{\infty}(-\xi\diff(\xi)+u,\diff(\xi))_{\spaceC}\,\xi\,\dinf\mu(\xi)-\left(u,\int_{0}^{\infty}(-\xi\diff(\xi)+u)\,\dinf\mu(\xi)\right)_{\spaceC}\right]\nonumber \\
 & =\Re\left[\int_{0}^{\infty}(-\xi\diff(\xi)+u,\xi\diff(\xi)-u)_{\spaceC}\,\dinf\mu(\xi)\right]\nonumber \\
 & =-\Re\left[\int_{0}^{\infty}\vert-\xi\diff(\xi)+u\vert^{2}\,\dinf\mu(\xi)\right]\leq0.\label{eq:=00005BEXTDIFF=00005D_Application-Passivity}
\end{alignat}

\end{enumerate}
\end{proof}
The remarks made for the standard case hold identically (in particular,
$\spaceD(C\&D)$ is nonempty). For $s\in\rho(A)$ we define
\begin{equation}
\hfun(s)\coloneqq s\,CR(s,A)B.\label{eq:=00005BEXTDIFF=00005D_Laplace}
\end{equation}

\subsection{Asymptotic stability\label{sub:=00005BEXTDIFF=00005D_Asymptotic-stability}}

Let $(A,B,C)$ be the triplet of operators defined in Section~\ref{sub:=00005BEXTDIFF=00005D_Abstract-realization},
further assumed to be non-null. The abstract Cauchy problem (\ref{eq:=00005BSTAB=00005D_Abstract-Cauchy-Problem})
considered herein is the following. The state space is
\begin{equation}
\begin{gathered}\spaceState\coloneqq\nabla H^{1}(\Omega)\times L^{2}(\Omega)\times L^{2}(\partial\Omega;\spaceDiff_{0}),\\
((\uac,\pac,\diff),(\srcuac,\srcpac,\srcdiff))_{\spaceState}\coloneqq(\uac,\srcuac)+(\pac,\srcpac)+(\diff,\srcdiff)_{L^{2}(\partial\Omega;\spaceDiff_{0})},
\end{gathered}
\label{eq:=00005BEXTDIFF=00005D_State-Space}
\end{equation}
and $\opA$ is defined as
\begin{equation}
\begin{gathered}\spaceDomain\ni\state\coloneqq\left(\begin{array}{c}
\uac\\
\pac\\
\diff
\end{array}\right)\longmapsto\opA\state\coloneqq\left(\begin{array}{c}
-\nabla\pac\\
-\opDiv\uac\\
A\diff+B\uac\cdot\normal
\end{array}\right),\\
\spaceDomain\coloneqq\left\{ (\uac,\pac,\diff)\in\spaceState\;\left|\;\begin{alignedat}{1} & (\uac,\pac)\in H_{\opDiv}(\Omega)\times H^{1}(\Omega)\\
 & (A\diff+B\uac\cdot\normal)\in L^{2}(\partial\Omega;\spaceDiff_{1})\\
 & \pac=C(A\diff+B\uac\cdot\normal)\;\text{in }H^{\frac{1}{2}}(\partial\Omega)
\end{alignedat}
\right.\right\} .
\end{gathered}
\label{eq:=00005BEXTDIFF=00005D_Definition-A}
\end{equation}
The technicality here is that the operator $(\diff,u)\mapsto C(A\diff+Bu)$
is defined over $\spaceD(C\&D)$, but $CB$ is not defined in general:
this is the abstract counterpart of (\ref{eq:=00005BEXTDIFF=00005D_Infinite-Integral}).
An immediate consequence of the definition of $\spaceDomain$ is given
in the following lemma.
\begin{lem}[Boundary regularity]
\label{lem:=00005BEXTDIFF=00005D_Boundary-Regularity-u}If $\state=(\uac,\pac,\diff)\in\spaceDomain$,
then $\uac\cdot\normal\in L^{2}(\partial\Omega)$.\end{lem}
\begin{proof}
Let $\state\in\spaceDomain$. By definition of $\spaceDomain$, we
have $\diff\in L^{2}(\partial\Omega;\spaceDiff_{0})$ so that $A\diff\in L^{2}(\partial\Omega;\spaceDiff_{-1})$
from Lemma~\ref{lem:=00005BEXTDIFF=00005D_Admissible}(iia) and Remark~\ref{rem:=00005BDELAY=00005D_Bochner-Integration}.
The proof is then identical to that of Lemma~\ref{lem:=00005BDIFF=00005D_Boundary-Regularity-u}.
\end{proof}
The application of Corollary~\ref{cor:=00005BSTAB=00005D_Asymptotic-Stability}
is summarized in the lemmas below, namely Lemmas~\ref{lem:=00005BEXTDIFF=00005D_Dissipativity},
\ref{lem:=00005BEXTDIFF=00005D_Injectivity}, and \ref{lem:=00005BEXTDIFF=00005D_Bijectivity}.
Due to the similarities with the standard case, the proofs are more
concise and focus on the differences.
\begin{lem}
\label{lem:=00005BEXTDIFF=00005D_Dissipativity}The operator $\opA$
defined by (\ref{eq:=00005BEXTDIFF=00005D_Definition-A}) is dissipative.\end{lem}
\begin{proof}
Let $\state\in\spaceDomain$. In particular, $\uac\cdot\normal\in L^{2}(\partial\Omega)$
from Lemma~\ref{lem:=00005BEXTDIFF=00005D_Boundary-Regularity-u}.
Green's formula (\ref{eq:=00005BPRE=00005D_Green-Formula}) and (\ref{eq:=00005BEXTDIFF=00005D_Definition-Dissipativity})
yield
\begin{alignat*}{1}
\Re(\opA\state,\state)_{\spaceState}=\Re\Bigl[(A\diff & +B\uac\cdot\normal,\diff)_{L^{2}(\partial\Omega;\spaceDiff_{0})}\\
 & -(\uac\cdot\normal,C(A\diff+B\uac\cdot\normal)){}_{L^{2}(\partial\Omega)}\Bigr]\leq0,
\end{alignat*}
using Lemma~\ref{lem:=00005BEXTDIFF=00005D_Admissible}.
\end{proof}
The next proof is much simpler than in the standard case.
\begin{lem}
\label{lem:=00005BEXTDIFF=00005D_Injectivity}$\opA$, given by (\ref{eq:=00005BEXTDIFF=00005D_Definition-A}),
is injective.\end{lem}
\begin{proof}
Assume $\state\in\spaceDomain$ satisfies $\opA\state=0$. In
particular $\nabla\pac=\vector 0$, $\opDiv\uac=0$,
and $A\diff+B\uac\cdot\normal=0$ in $L^{2}(\partial\Omega;\spaceDiff_{1})$.
The IBC (i.e. the third equation in $\spaceDomain$) gives $\pac=0$
in $L^{2}(\partial\Omega)$ hence $p=0$ in $L^{2}(\Omega)$. From
Lemma~\ref{lem:=00005BEXTDIFF=00005D_Boundary-Regularity-u}, $\uac\cdot\normal\in L^{2}(\partial\Omega)$
so we have at least $B\uac\cdot\normal\in L^{2}(\partial\Omega;\spaceDiff_{-1})$.
Using (\ref{eq:=00005BEXTDIFF=00005D_Definition-Injectivity}), we
deduce $\diff=0$ and $\uac\cdot\normal=0$, hence $\uac=0$ from
(\ref{eq:=00005BPRE=00005D_Hodge-Decomposition}).\end{proof}
\begin{lem}
\label{lem:=00005BEXTDIFF=00005D_Bijectivity}$s\opId-\opA$, with
$\opA$ given by (\ref{eq:=00005BEXTDIFF=00005D_Definition-A}), is
bijective for $s\in(0,\infty)\cup i\spaceR^{*}$.\end{lem}
\begin{proof}
Let $\srcstate\in\spaceState$, $s\in(0,\infty)\cup i\spaceR^{*}$,
and $\test\in H^{1}(\Omega)$. We seek a \emph{unique} $\state\in\spaceDomain$
such that $(s\opId-\opA)\state=\srcstate$, i.e. (\ref{eq:=00005BDIFF=00005D_Bijectivity-Eq}),
which implies
\begin{equation}
\begin{alignedat}{1}(\nabla\pac,\nabla\test)+s^{2}(\pac,\test)+\frac{s}{\zfun(s)}(\pac,\test)_{L^{2}(\partial\Omega)}= & (\srcuac,\nabla\test)+s(\srcpac,\test)\\
 & +\frac{s}{\zfun(s)}(CAR(s,A)\srcdiff,\test)_{L^{2}(\partial\Omega)}.
\end{alignedat}
\label{eq:=00005BEXTDIFF=00005D_Bijectivity-WeakForm-p}
\end{equation}
Note that, from (\ref{eq:=00005BEXTDIFF=00005D_Definition-Regularity}),
the right-hand side defines an anti-linear form on
$H^{1}(\Omega)$. Let us denote by $\pac$ the unique solution of
(\ref{eq:=00005BEXTDIFF=00005D_Bijectivity-WeakForm-p}) obtained
from a pointwise application of Theorem~\ref{thm:=00005BMOD=00005D_Weak-Form-p}
(we rely here on (\ref{eq:=00005BDIFF=00005D_Definition-Reality})).
It remains to find suitable $\uac$ and $\diff$, in a manner identical
to the standard diffusive case.

Taking $\test\in\spaceContinuous_{0}^{\infty}(\Omega)$ in (\ref{eq:=00005BEXTDIFF=00005D_Bijectivity-WeakForm-p})
shows that $\uac\in H_{\opDiv}(\Omega)$ with (\ref{eq:=00005BDIFF=00005D_Bijectivity-Eq}b).
Using the expressions of $\uac\in\nabla H^{1}(\Omega)$ and $\opDiv\uac$,
and Green's formula (\ref{eq:=00005BPRE=00005D_Green-Formula}), the
weak formulation~(\ref{eq:=00005BEXTDIFF=00005D_Bijectivity-WeakForm-p})
shows that $\pac$ and $\uac$ satisfy, in $L^{2}(\partial\Omega)$,
\begin{equation}
\pac=\zfun(s)\uac\cdot\normal+CAR(s,A)\srcdiff.\label{eq:=00005BEXTDIFF=00005D_Bijectivity-IBC}
\end{equation}

Let us now \emph{define} $\diff$ as
\[
\diff\coloneqq R(s,A)\left(B\uac\cdot\normal+\srcdiff\right)\in L^{2}(\partial\Omega;\spaceDiff_{0}).
\]
Using the property (\ref{eq:=00005BEXTDIFF=00005D_Definition-Regularity}),
we obtain that
\begin{alignat*}{1}
A\diff+B\uac\cdot\normal & =AR(s,A)\left(B\uac\cdot\normal+\srcdiff\right)+B\uac\cdot\normal\\
 & =sR(s,A)B\uac\cdot\normal+AR(s,A)\srcdiff
\end{alignat*}
belongs to $L^{2}(\partial\Omega;\spaceDiff_{1})$. We show that the
IBC holds by rewriting (\ref{eq:=00005BEXTDIFF=00005D_Bijectivity-IBC})
as
\begin{alignat*}{1}
\pac & =C(sR(s,A)B\uac\cdot\normal+AR(s,A)\srcdiff)\\
 & =C(AR(s,A)B\uac\cdot\normal+B\uac\cdot\normal+AR(s,A)\srcdiff)\\
 & =C(A\diff+B\uac\cdot\normal),
\end{alignat*}
using (\ref{eq:=00005BEXTDIFF=00005D_Laplace}). Thus $(\uac,\pac,\diff)\in\spaceDomain$.
The uniqueness of $\pac$ follows from Theorem~\ref{thm:=00005BMOD=00005D_Weak-Form-p},
that of $\uac$ from (\ref{eq:=00005BPRE=00005D_Hodge-Decomposition}),
and that of $\diff$ from $s\in\rho(A)$.
\end{proof}

\section{Addition of a derivative term\label{sec:=00005BDER=00005D_Addition-of-Derivative}}

By \emph{derivative impedance} we mean
\[
\hat{z}(s)=z_{1}s,\;z_{1}>0,
\]
for which the IBC (\ref{eq:=00005BMOD=00005D_IBC}) reduces to $\pac=z_{1}\partial_{t}\uac\cdot\normal.$

The purpose of this section is to illustrate, on two
examples, that the addition of such a derivative term to the IBCs
covered so far (\ref{eq:=00005BDELAY=00005D_Laplace},\ref{eq:=00005BDIFF=00005D_Laplace-Standard-Diffusive},\ref{eq:=00005BEXTDIFF=00005D_Laplace-Extended-Diffusive})
leaves unchanged the asymptotic stability results obtained with Corollary~\ref{cor:=00005BSTAB=00005D_Asymptotic-Stability}:
it only makes the proofs more cumbersome as the state space becomes
lengthier. This is why this term has not been included
in Sections~\ref{sec:=00005BDELAY=00005D_Delay-impedance}--\ref{sec:=00005BEXTDIFF=00005D_Extended-diffusive-impedance}.

The examples will also illustrate why establishing the
asymptotic stability of (\ref{eq:=00005BMOD=00005D_Wave-Equation},\ref{eq:=00005BMOD=00005D_IBC})
with (\ref{eq:=00005BMOD=00005D_Target-Impedance}) can be done by
treating each positive-real term in (\ref{eq:=00005BMOD=00005D_Target-Impedance})
separately (i.e. by building the realization of each of the four positive-real
term separately and then aggregating them), thus justifying a posteriori
the structure of the article.

\begin{example}[Proportional-derivative impedance]
Consider the following positive-real impedance kernel
\begin{equation}
\hat{z}(s)=z_{0}+z_{1}s,\label{eq:=00005BDER=00005D_Prop-Der-Impedance}
\end{equation}
where $z_{0},z_{1}>0$. The energy space is
\[
\begin{gathered}\spaceState\coloneqq\nabla H^{1}(\Omega)\times L^{2}(\Omega)\times L^{2}(\partial\Omega),\\
\begin{alignedat}{1}((\uac,\pac,\deriv),(\srcuac,\srcpac,\srcderiv))_{\spaceState}\coloneqq(\uac,\srcuac)+(\pac,\srcpac)+z_{1}(\deriv,\srcderiv)_{L^{2}(\partial\Omega)}\end{alignedat}
\end{gathered}
\]
and the corresponding evolution operator is
\[
\begin{gathered}\spaceDomain\ni\state\coloneqq\left(\begin{array}{c}
\uac\\
\pac\\
\deriv
\end{array}\right)\longmapsto\opA\state\coloneqq\left(\begin{array}{c}
-\nabla\pac\\
-\opDiv\uac\\
\frac{1}{z_{1}}\left[\pac-z_{0}\uac\cdot\normal\right]
\end{array}\right),\\
\spaceDomain\coloneqq\left\{ (\uac,\pac,\deriv)\in\spaceState\;\left|\;\begin{alignedat}{1} & (\uac,\pac)\in H_{\opDiv}(\Omega)\times H^{1}(\Omega)\\
 & \deriv=\uac\cdot\normal\;\text{in }L^{2}(\partial\Omega)
\end{alignedat}
\right.\right\} .
\end{gathered}
\]
Note how the derivative term in (\ref{eq:=00005BDER=00005D_Prop-Der-Impedance})
is accounted for by adding the state variable $\eta\in L^{2}\left(\partial\Omega\right)$.
The application of Corollary~\ref{cor:=00005BSTAB=00005D_Asymptotic-Stability}
is straightforward. For instance, for $\state\in\spaceDomain$, we
have
\begin{alignat*}{1}
\Re(\opA\state,\state)_{\spaceState}= & -\Re\left[(\uac\cdot\normal,\pac)_{L^{2}(\partial\Omega)}\right]+\Re\left[(\pac-z_{0}\uac\cdot\normal,\deriv)_{L^{2}(\partial\Omega)}\right]\\
= & -z_{0}\Vert\uac\cdot\normal\Vert_{L^{2}(\partial\Omega)}^{2},
\end{alignat*}
so that $\opA$ is dissipative. The injectivity of $\opA$ and the
bijectivity of $s\opId-\opA$ for $s\in(0,\infty)\cup i\spaceR^{*}$
can be proven similarly to what has been done in the previous sections.
\end{example}

\begin{example}
Let us revisit the delay impedance (\ref{eq:=00005BDELAY=00005D_Laplace}),
covered in Section~\ref{sec:=00005BDELAY=00005D_Delay-impedance},
by adding a derivative term to it:
\begin{equation}
\hat{z}(s)\coloneqq z_{1}s+z_{0}+z_{\tau}e^{-\tau s},\label{eq:=00005BDER=00005D_Impedance-derivative}
\end{equation}
where $z_{1}>0$ and $(z_{0},z_{\tau})$ are defined as in Section~\ref{sec:=00005BDELAY=00005D_Delay-impedance},
so that $\hat{z}$ is positive-real. The inclusion of the derivative
implies the presence of an additional variable in the extended state,
i.e. the state space is (compare with (\ref{eq:=00005BDELAY=00005D_Scalar-Prod}))
\[
\begin{gathered}\spaceState\coloneqq\nabla H^{1}(\Omega)\times L^{2}(\Omega)\times L^{2}(\partial\Omega;L^{2}(-\tau,0))\times L^{2}(\partial\Omega),\\
\begin{alignedat}{1}((\uac,\pac,\delay,\deriv),(\srcuac,\srcpac,\srcdelay,\srcderiv))_{\spaceState}\coloneqq(\uac,\srcuac)+(\pac,\srcpac)+k(\delay,\srcdelay & )_{L^{2}(\partial\Omega;L^{2}(-\tau,0))}\\
 & +z_{1}(\deriv,\srcderiv)_{L^{2}(\partial\Omega)}.
\end{alignedat}
\\
\\
\end{gathered}
\]
The operator $\opA$ becomes (compare with (\ref{eq:=00005BDELAY=00005D_Definition-A}))
\[
\begin{gathered}\spaceDomain\ni\state\coloneqq\left(\begin{array}{c}
\uac\\
\pac\\
\delay\\
\deriv
\end{array}\right)\longmapsto\opA\state\coloneqq\left(\begin{array}{c}
-\nabla\pac\\
-\opDiv\uac\\
\partial_{\theta}\delay\\
\frac{1}{z_{1}}\left[\pac-z_{0}\uac\cdot\normal-z_{\tau}\delay(\cdot,-\tau)\right]
\end{array}\right),\\
\spaceDomain\coloneqq\left\{ (\uac,\pac,\delay,\deriv)\in\spaceState\;\left|\;\begin{alignedat}{1} & (\uac,\pac,\delay)\in H_{\opDiv}(\Omega)\times H^{1}(\Omega)\times L^{2}(\partial\Omega;H^{1}(-\tau,0))\\
 & \delay(\cdot,0)=\uac\cdot\normal\;\text{in }L^{2}(\partial\Omega)\\
 & \deriv=\uac\cdot\normal\;\text{in }L^{2}(\partial\Omega)
\end{alignedat}
\right.\right\} ,
\end{gathered}
\]
where the IBC (\ref{eq:=00005BMOD=00005D_IBC},\ref{eq:=00005BDER=00005D_Impedance-derivative})
is the third equation in $\spaceDomain$. The application of Corollary~\ref{cor:=00005BSTAB=00005D_Asymptotic-Stability}
is identical to Section~\ref{sub:=00005BDELAY=00005D_Asymptotic-stability}.
For instance, for $\state\in\spaceDomain$, we have
\begin{alignat*}{1}
\Re(\opA\state,\state)_{\spaceState}= & -\Re\left[(\uac\cdot\normal,\pac)_{L^{2}(\partial\Omega)}\right]+\Re\left[k(\partial_{\theta}\delay,\delay)_{L^{2}(\partial\Omega;L^{2}(-\tau,0))}\right]\\
 & +\Re\left[(\pac-z_{0}\uac\cdot\normal-z_{\tau}\delay(\cdot,-\tau),\deriv)_{L^{2}(\partial\Omega)}\right]\\
= & -\Re\left[(\uac\cdot\normal,\pac)_{L^{2}(\partial\Omega)}\right]+\frac{k}{2}\Re\left[\Vert\uac\cdot\normal\Vert_{L^{2}(\partial\Omega)}^{2}-\Vert\delay(\cdot,-\tau)\Vert_{L^{2}(\partial\Omega)}^{2}\right]\\
 & +\Re\left[(\pac-z_{0}\uac\cdot\normal-z_{\tau}\delay(\cdot,-\tau),\uac\cdot\normal)_{L^{2}(\partial\Omega)}\right]\\
= & \left(\frac{k}{2}-z_{0}\right)\Vert\uac\cdot\normal\Vert_{L^{2}(\partial\Omega)}^{2}-\frac{k}{2}\Vert\delay(\cdot,-\tau)\Vert_{L^{2}(\partial\Omega)}^{2}\\
 & -z_{\tau}\Re\left[(\delay(\cdot,-\tau),\uac\cdot\normal)_{L^{2}(\partial\Omega)}\right],
\end{alignat*}
so that the expression of $\Re(\opA\state,\state)_{\spaceState}$
is identical to that without a derivative term, see the proof of Lemma~\ref{lem:=00005BDELAY=00005D_Dissipativity}.
The proof of the injectivity of $\opA$ is also identical to that
carried out in Lemma~\ref{lem:=00005BDELAY=00005D_Injectivity}:
the condition $\opA\state=0$ yields $\delay(\cdot,0)=\delay(\cdot,-\tau)=\uac\cdot\normal=\deriv$
a.e. on $\partial\Omega$. Finally, the proof of Lemma~\ref{lem:=00005BDELAY=00005D_Bijectivity}
can also be followed almost identically to solve $(s\opId-\opA)\state=\srcstate$
with $\srcstate=(\srcuac,\srcpac,\srcdelay,\srcderiv)$, the additional
steps being straightforward; after defining uniquely $\pac$, $\uac$,
and $\delay$, the only possibility for $\deriv$ is $\deriv\coloneqq\uac\cdot\normal$,
which belongs to $L^{2}(\partial\Omega)$, and $\deriv=\delay(\cdot,0)$
is deduced from (\ref{eq:=00005BDELAY=00005D_Bijectivity-Delay}).
\end{example}

\section{Conclusions and perspectives}

This paper has focused on the asymptotic stability of the wave equation
coupled with positive-real IBCs drawn from physical applications,
namely time-delayed impedance in Section~\ref{sec:=00005BDELAY=00005D_Delay-impedance},
standard diffusive impedance (e.g. fractional integral) in Section~\ref{sec:=00005BDIFF=00005D_Standard-diffusive-impedance},
and extended diffusive impedance (e.g. fractional derivative) in Section~\ref{sec:=00005BEXTDIFF=00005D_Extended-diffusive-impedance}.
Finally, the invariance of the derived asymptotic stability results
under the addition of a derivative term in the impedance has been
discussed in Section~\ref{sec:=00005BDER=00005D_Addition-of-Derivative}.
The proofs crucially hinge upon the knowledge of a dissipative realization
of the IBC, since it employs the semigroup asymptotic stability result
given in \cite{arendt1988tauberian,lyubich1988asymptotic}.

By combining these results, asymptotic stability is obtained for the
impedance $\hat{z}$ introduced in Section~\ref{sec:=00005BMOD=00005D_Model-and-preliminary-results}
and given by (\ref{eq:=00005BMOD=00005D_Target-Impedance}). This
suggests the first perspective of this work, formulated as a conjecture.
\begin{conjecture}
Assume $\hat{z}$ is positive-real, without isolated singularities
on $i\spaceR$. Then the Cauchy problem (\ref{eq:=00005BMOD=00005D_Wave-Equation},\ref{eq:=00005BMOD=00005D_IBC})
is asymptotically stable in a suitable energy space.
\end{conjecture}
Establishing this conjecture using the method of proof
used in this paper first requires building a dissipative realization
of the impedance operator $u\mapsto z\star u$.

If $\hat{z}$ is assumed rational and proper (i.e. $\hat{z}(\infty)$
is finite), a dissipative realization can be obtained using the celebrated
positive-real lemma, also known as the Kalman\textendash Yakubovich\textendash Popov
lemma \cite[Thm.~3]{anderson1967prmatrices}; the proof of asymptotic
stability is then a simpler version of that carried out in Section~\ref{sec:=00005BDIFF=00005D_Standard-diffusive-impedance},
see \cite[\S\,4.3]{monteghetti2018dissertation} for the details.
If $\hat{z}$ is not proper, it can be written as $\hat{z}=a_{1}s+\hat{z}_{\text{p}}$
where $a_{1}>0$ and $\hat{z}_{\text{p}}$ is proper (see Remark~\ref{rem:=00005BMOD=00005D_PR-growth-at-infinity});
each term can be covered separately, see Section~\ref{sec:=00005BDER=00005D_Addition-of-Derivative}.

If $\hat{z}$ is not rational, then a suitable infinite-dimensional
variant of the positive-real lemma is required. For
instance, \cite[Thm.~5.3]{staffans2002passive} gives a realization
using system nodes; a difficulty in using this result is that the
properties needed for the method of proof presented here do not seem
to be naturally obtained with system nodes. This result would be
sharp, in the sense that it is known that exponential stability is
not achieved in general (consider for instance $\hat{z}(s)=\nicefrac{1}{\sqrt{s}}$
that induces an essential spectrum with accumulation point at $0$).
If this conjecture proves true, then the rate of decay of the solution
could also be studied and linked to properties of the impedance $\hat{z}$;
this could be done by adapting the techniques used in
\cite{stahn2018waveequation}.

To illustrate this conjecture, let us give two examples of positive-real
impedance kernels that are \emph{not} covered by the results of this
paper. Both examples arise in physical applications \cite{monteghetti2016diffusive}
and have been used in numerical simulations \cite{monteghetti2017tdibc}.
The first example is a kernel similar to (\ref{eq:=00005BMOD=00005D_Target-Impedance}),
namely
\[
\hat{z}(s)=z_{0}+z_{\tau}e^{-\tau s}+z_{1}s+\int_{0}^{\infty}\frac{\mu(\xi)}{s+\xi}\,\dinf\xi\quad\left(\Re(s)>0\right),
\]
where $\tau>0$, $z_{\tau}\in\spaceR$, $z_{0}\geq\vert z_{\tau}\vert$,
$z_{1}>0$, and the weight $\mu\in\spaceContinuous^{\infty}((0,\infty))$
satisfies the condition $\int_{0}^{\infty}\frac{\vert\mu(\xi)\vert}{1+\xi}\dinf\xi<\infty$
and is such that $\hat{z}$ is positive-real. When the sign of $\mu$
is indefinite the passivity condition (\ref{eq:=00005BDIFF=00005D_Application-Passivity})
does not hold, so that this impedance is not covered by the presented
results despite the fact that, overall, $\hat{z}$ is positive-real
with a realization formally identical to that of the impedance (\ref{eq:=00005BMOD=00005D_Target-Impedance})
defined in Section~\ref{sec:=00005BMOD=00005D_Model-and-preliminary-results}.

The second and last example is
\[
\hat{z}(s)=z_{0}+z_{\tau}\frac{e^{-\tau s}}{\sqrt{s}},
\]
with $z_{\tau}\geq0$, $\tau>0$, and $z_{0}\geq0$ sufficiently large
for $\hat{z}$ to be positive-real (the precise condition is $z_{0}\geq-z_{\tau}\cos(\tilde{x}+\frac{\pi}{4})\sqrt{\nicefrac{\tau}{\tilde{x}}}$
where $\tilde{x}\simeq2.13$ is the smallest positive root of $x\mapsto\tan(x+\nicefrac{\pi}{4})+\nicefrac{1}{2x}$).
A simple realization can be obtained by combining Sections~\ref{sec:=00005BDELAY=00005D_Delay-impedance}
and \ref{sec:=00005BDIFF=00005D_Standard-diffusive-impedance}, i.e.
by delaying the diffusive representation using a transport equation:
the convolution then reads, for a causal input $u$,
\[
z\star u=z_{0}u+z_{\tau}\int_{0}^{\infty}\delay(t,-\tau,\xi)\,\dinf\mu(\xi),
\]
where $\diff$ and $\mu$ are defined as in Section~\ref{sec:=00005BDIFF=00005D_Standard-diffusive-impedance},
and for a.e. $\xi\in(0,\infty)$ the function $\delay(\cdot,\cdot,\xi)$
obeys the transport equation (\ref{eq:=00005BDELAY=00005D_Transport-Equation}ab)
but with $\delay(t,0,\xi)=\diff(t,\xi)$. So far, the authors have
not been able to find a suitable Lyapunov functional (i.e. a suitable
definition of $\Vert\cdot\Vert_{\spaceState}$) for this realization.

The second open problem we wish to point out is the extension of the
stability result to discontinuous IBCs. A typical case is a split
of the boundary $\partial\Omega$ into three disjoint parts: a Neumann
part $\partial\Omega_{N}$, a Dirichlet part $\partial\Omega_{D}$,
and an impedance part $\partial\Omega_{z}$ where one of the IBCs covered
in the paper is applied. Dealing with such discontinuities may involve
the redefinition of both the energy space $\spaceState$ and domain
$\spaceDomain$, as well as the derivation of compatibility constraints.
The proofs may benefit from considering the scattering formulation,
recalled in Remark~\ref{rem:=00005BMOD=00005D_Terminology}, which
enables to write the three boundary conditions in a unified fashion.

\section*{Acknowledgments}

This research has been financially supported by the French ministry
of defense (Direction G\'{e}n\'{e}rale de l'Armement) and ONERA
(the French Aerospace Lab). We thank the two referees for their helpful
comments. The authors are grateful to Prof. Patrick
Ciarlet for suggesting the use of the extension by zero in the proof
of Proposition~\ref{prop:=00005BMOD=00005D_Rellich-Lemma}.

\appendix

\section{Miscellaneous results}

For the sake of completeness, the key technical results upon which
the paper depends are briefly gathered here.

\subsection{Extension by zero\label{sub:=00005BMISC=00005D_Extension-by-zero}}

Let us define the zero extension operator as
\[
E:\,L^{2}(\Omega_{1})\rightarrow L^{2}(\Omega_{2}),\;Eu\coloneqq\left\{ u\;\text{on}\;\Omega_{1},\;0\;\text{on}\;\Omega_{2}\backslash\Omega_{1},\right.
\]
where $\Omega_{1}$ and $\Omega_{2}$ are two open subsets of $\spaceR^{d}$
such that $\overline{\Omega_{1}}\subset\Omega_{2}$.
\begin{prop}
\label{prop:Extension-by-zero}Let $\Omega_{1}$ and $\Omega_{2}$
be two bounded open subsets of $\spaceR^{d}$ such that $\overline{\Omega_{1}}\subset\Omega_{2}$.
For any $\pac\in H_{0}^{1}\left(\Omega_{1}\right)$, $E\pac\in H_{0}^{1}\left(\Omega_{2}\right)$
with
\begin{equation}
\forall i\in\left\llbracket 1,d\right\rrbracket ,\;\partial_{i}\left[E\pac\right]=E\left[\partial_{i}\pac\right]\;\text{a.e. in }\;\Omega_{2}.\label{eq:Extension-by-zero_Derivatives}
\end{equation}
In particular, $\Vert\pac\Vert_{H^{1}\left(\Omega_{1}\right)}=\Vert E\pac\Vert_{H^{1}\left(\Omega_{2}\right)}$.\end{prop}
\begin{rem*}
Note that we do not require any regularity on the boundary of $\Omega_{i}$.
This is due to the fact that the proof only relies on the definition
of $H_{0}^{1}$ by density.\end{rem*}
\begin{proof}
The first part of the proof is adapted from \cite[Lem.\,3.22]{adams1975Sobolev}.
By definition of $H_{0}^{1}\left(\Omega_{1}\right)$, there is a sequence
$\phi_{n}\in\spaceContinuous_{0}^{\infty}\left(\Omega_{1}\right)$
converging to $\pac$ in the $\Vert\cdot\Vert_{H^{1}\left(\Omega_{1}\right)}$
norm. Since $E\pac\in L^{2}\left(\Omega_{2}\right)$, $E\pac$ is
locally integrable and thus belongs to $\spaceD^{'}\left(\Omega_{2}\right)$.
For any $\varphi\in\spaceContinuous_{0}^{\infty}\left(\Omega_{2}\right)$,
we have \begingroup \allowdisplaybreaks
\begin{alignat*}{2}
\left\langle \partial_{i}\left[E\pac\right],\varphi\right\rangle _{\spaceD^{'}\left(\Omega_{2}\right),\spaceContinuous_{0}^{\infty}\left(\Omega_{2}\right)} & \coloneqq-\left\langle E\pac,\partial_{i}\varphi\right\rangle _{\spaceD^{'}\left(\Omega_{2}\right),\spaceContinuous_{0}^{\infty}\left(\Omega_{2}\right)}\\
 & =-\int_{\Omega_{1}}\pac\partial_{i}\varphi & \qquad\ensuremath{\left(E\pac\in L^{2}\left(\Omega_{2}\right)\right)}\\
 & =-\lim_{n\rightarrow\infty}\int_{\Omega_{1}}\phi_{n}\partial_{i}\varphi & \qquad\ensuremath{\left(\pac\in H_{0}^{1}\left(\Omega_{1}\right)\right)}\\
 & =\lim_{n\rightarrow\infty}\int_{\Omega_{1}}\partial_{i}\phi_{n}\varphi & \qquad\ensuremath{\left(\phi_{n}\in\spaceContinuous_{0}^{\infty}\left(\Omega_{1}\right)\right)}\\
 & =\int_{\Omega_{1}}\partial_{i}\pac\varphi & \qquad\left(\partial_{i}\phi_{n}\xrightarrow[n\rightarrow\infty]{L^{2}\left(\Omega_{1}\right)}\partial_{i}\pac\right)\\
 & =\int_{\Omega_{2}}E\left[\partial_{i}\pac\right]\varphi,
\end{alignat*}
\endgroup hence $E\left[\partial_{i}\pac\right]=\partial_{i}\left[E\pac\right]$
in $\spaceD^{'}\left(\Omega_{2}\right)$. Since $\partial_{i}\pac\in L^{2}\left(\Omega_{1}\right)$
by assumption, we deduce from this identity that $E\left[\partial_{i}\pac\right]\in L^{2}\left(\Omega_{2}\right)$.
Hence $E\pac\in H^{1}\left(\Omega_{2}\right)$.

Using the fact that $E$ is an isometry from $H_{0}^{1}\left(\Omega_{1}\right)$
to $H^{1}\left(\Omega_{2}\right)$ we deduce
\[
\Vert E\phi_{n}-E\pac\Vert_{H^{1}\left(\Omega_{2}\right)}=\Vert E\left(\phi_{n}-\pac\right)\Vert_{H^{1}\left(\Omega_{2}\right)}=\Vert\phi_{n}-\pac\Vert_{H^{1}\left(\Omega_{1}\right)}\xrightarrow[n\rightarrow\infty]{}0.
\]
Since $E\phi_{n}\in\spaceContinuous_{0}^{\infty}\left(\Omega_{2}\right)$,
this shows that $E\pac\in H_{0}^{1}\left(\Omega_{2}\right)$.
\end{proof}

\subsection{Compact embedding and trace operator\label{sub:=00005BMISC=00005D_Embedding-Trace}}

Let $\Omega\subset\spaceR^{d}$, $d\in\llbracket1,\infty\llbracket$,
be a bounded open set with a Lipschitz boundary.

The embedding $H^{1}(\Omega)\subset H^{s}(\Omega)$ with $s\in[0,1)$
is compact \cite[Thm.~1.4.3.2]{grisvard2011elliptic}. (See \cite[Thm.~16.17]{lionsMagenes1972BVP1}
for smooth domains.)

The trace operator $H^{s}(\Omega)\rightarrow H^{s-\nicefrac{1}{2}}(\partial\Omega)$
with $s\in(\nicefrac{1}{2},1]$ is continuous and surjective \cite[Thm.~1.5.1.2]{grisvard2011elliptic}.
(See \cite[Thm.~1]{ding1996trace} if $\Omega$ is also simply connected
and \cite[Thm.~9.4]{lionsMagenes1972BVP1} for smooth domains.)

The trace operator $H_{\opDiv}(\Omega)\rightarrow H^{-\frac{1}{2}}(\partial\Omega)$,
$\uac\mapsto\uac\cdot\normal$ is continuous \cite[Thm.~2.5]{giraultraviart1986femNavierStokes},
and the following Green's formula holds for $\test\in H^{1}(\Omega)$
\cite[Eq.~(2.17)]{giraultraviart1986femNavierStokes}
\begin{equation}
(\uac,\nabla\test)+(\opDiv\uac,\test)=\langle\uac\cdot\normal,\overline{\test}\rangle_{H^{-\frac{1}{2}}(\partial\Omega),H^{\frac{1}{2}}(\partial\Omega)}.\label{eq:=00005BPRE=00005D_Green-Formula}
\end{equation}

\subsection{Hodge decomposition\label{sub:=00005BMISC=00005D_Hodge-decomposition}}

Let $\Omega\subset\spaceR^{d}$, $d\in\llbracket1,\infty\llbracket$,
be a connected open set with a Lipschitz boundary. The following orthogonal
decomposition holds \cite[Prop.~IX.1]{dautraylions1980vol3spectral}
\begin{equation}
(L^{2}(\Omega))^{d}=\nabla H^{1}(\Omega)\varoplus H_{\opDiv0,0}(\Omega),\label{eq:=00005BPRE=00005D_Hodge-Decomposition}
\end{equation}
where
\[
\nabla H^{1}(\Omega)\coloneqq\left\{ \vector f\in(L^{2}(\Omega))^{d}\;\vert\;\exists g\in H^{1}(\Omega):\;\vector f=\nabla g\right\}
\]
is a closed subspace of $(L^{2}(\Omega))^{d}$ and
\[
H_{\opDiv0,0}(\Omega)\coloneqq\left\{ \vector f\in H_{\opDiv}(\Omega)\;\vert\;\opDiv\vector f=0,\;\vector f\cdot\normal=0\;\text{in }H^{-\frac{1}{2}}(\partial\Omega)\right\} .
\]
This result still holds true when $\Omega$ is disconnected
(the proof of \cite[Prop.~IX.1]{dautraylions1980vol3spectral} relies
on Green's formula (\ref{eq:=00005BPRE=00005D_Green-Formula}) as
well as the compactness of the embedding $H^{1}\left(\Omega\right)\subset L^{2}\left(\Omega\right)$,
needed to apply Peetre's lemma).
\begin{rem}
The space $H_{\opDiv0,0}(\Omega)$ is studied in \cite[Chap.~IX]{dautraylions1980vol3spectral}
for $n=2$ or $3$. For instance,
\[
\mathbb{H}_{1}\coloneqq H_{\opDiv0,0}(\Omega)\cap\left\{ \vector f\in(L^{2}(\Omega))^{d}\;\vert\;\nabla\times\vector f=\vector 0\right\}
\]
has a finite dimension under suitable assumptions on the set $\Omega$
\cite[Prop.~IX.2]{dautraylions1980vol3spectral}.
\end{rem}

\subsection{Semigroups of linear operators\label{sub:=00005BMISC=00005D_Asymptotic-stability-of-Semigroup}}
\begin{thm}[Lumer-Phillips]
\label{thm:=00005BSTAB=00005D_Lumer-Phillips}Let $\spaceState$
be a complex Hilbert space and $\opA:\,\spaceDomain\subset\spaceState\rightarrow\spaceState$
an unbounded operator. If $\Re(\opA\state,\state)_{H}\leq0$ for every
$\state\in\spaceDomain$ and $\opId-\opA$ is surjective, then $\opA$
is the infinitesimal generator of a strongly continuous semigroup
of contractions $\opT(t)\in\spaceBounded(\spaceState)$.\end{thm}
\begin{proof}
The result follows from \cite[Thms.~4.3 \& 4.6]{pazy1983stability}
since Hilbert spaces are reflexive \cite[Thm.~8.9]{lax2002funana}. \end{proof}
\begin{thm}[Asymptotic stability \cite{arendt1988tauberian,lyubich1988asymptotic}]
\label{thm:=00005BSTAB=00005D_Arendt-Batty} Let $\spaceState$ be
a complex Hilbert space and $\opA:\,\spaceDomain\subset\spaceState\rightarrow\spaceState$
be the infinitesimal generator of a strongly continuous semigroup
$\opT(t)\in\spaceBounded(\spaceState)$ of contractions. If $\sigma_{p}(\opA)\cap i\spaceR=\varnothing$
and $\sigma(\opA)\cap i\spaceR$ is countable, then $\opT$ is asymptotically
stable, i.e. $\opT(t)\state_{0}\rightarrow0$ in $\spaceState$ as
$t\rightarrow\infty$ for any $\state_{0}\in\spaceState$.
\end{thm}

\section{Application of the invariance principle\label{sub:=00005BSTAB=00005D_Invariance-Principle}}

The purpose of this appendix is to justify why, in this paper, we
rely on Corollary~\ref{cor:=00005BSTAB=00005D_Asymptotic-Stability}
rather than the invariance principle, commonly used with dynamical
systems on Banach spaces. Theorem~\ref{thm:=00005BSTAB=00005D_Invariance-Principle}
below states the invariance principle for the case of interest herein,
i.e. a linear Cauchy problem (\ref{eq:=00005BSTAB=00005D_Abstract-Cauchy-Problem})
for which the Lyapunov functional is $\frac{1}{2}\Vert\cdot\Vert_{\spaceState}^{2}$.
(For further background, see \cite[\S~3.7]{luoGuoMorgul2012stability}
and \cite[Chap.~9]{cazenave1998evolution}.)
\begin{thm}[Invariance principle]
\label{thm:=00005BSTAB=00005D_Invariance-Principle} Let $\opA$
be the infinitesimal generator of a strongly continuous semigroup
of contractions $\opT(t)\in\spaceBounded(\spaceState)$ and $\state_{0}\in\spaceState$.
If the orbit $\gamma(\state_{0})\coloneqq\bigcup_{t\geq0}\opT(t)\state_{0}$
lies in a compact set of $\spaceState$, then $\opT(t)\state_{0}\rightarrow M$
as $t\rightarrow\infty$, where $M$ is the largest $\opT$-invariant
set in
\begin{equation}
\left\{ \state\in\spaceDomain\;\left|\;\Re\left[(\opA\state,\state)_{\spaceState}\right]=0\right.\right\} .\label{eq:=00005BSTAB=00005D_IVP-Stationary-Set}
\end{equation}
\end{thm}
\begin{proof}
The function $\Lyapunov\coloneqq\frac{1}{2}\Vert\cdot\Vert_{\spaceState}^{2}$
is continuous on $\spaceState$ and satisfies $\Lyapunov(\opT(t)\state)\leq\Lyapunov(\state)$
for any $\state\in\spaceState$ so that it is a Lyapunov functional.
The invariance principle \cite[Thm.~1]{hale1969invarianceprinciple}
then shows that $\opT(t)\state_{0}$ is attracted to the largest invariant
set of
\[
\left\{ \state\in\spaceState\;\left|\;\lim_{t\rightarrow0^{+}}t^{-1}(\Lyapunov(\opT(t)\state)-\Lyapunov(\state))=0\right.\right\} .
\]

\end{proof}
Let us now discuss the application of Theorem~\ref{thm:=00005BSTAB=00005D_Invariance-Principle}
to (\ref{eq:=00005BMOD=00005D_Wave-Equation},\ref{eq:=00005BMOD=00005D_IBC}),
assuming we know a dissipative realization of the impedance operator
$u\mapsto z\star u$ in a state space $\spaceDiff_{0}$.

The first step is to establish that the largest invariant subset of
(\ref{eq:=00005BSTAB=00005D_IVP-Stationary-Set}) reduces to $\{0\}$,
i.e. that the only solution of (\ref{eq:=00005BSTAB=00005D_Abstract-Cauchy-Problem})
in (\ref{eq:=00005BSTAB=00005D_IVP-Stationary-Set}) is null,
which is verified by the evolution operators defined in Sections~\ref{sec:=00005BDELAY=00005D_Delay-impedance}--\ref{sec:=00005BDER=00005D_Addition-of-Derivative}.
This requires to exclude solenoidal fields from $\state_{0}$, see
Remark~\ref{rem:|Delay=00005D_Exclusion-Solenoidal-Fields}.

The second step is to prove the precompactness of the orbit $\gamma(\state_{0})$
for any $\state_{0}$ in $\spaceState$. The following criterion can
be used, where for $s\in\rho(\opA)$ we denote the resolvent operator
by
\begin{equation}
R(s,\opA)\coloneqq(s\opId-\opA)^{-1}.\label{eq:=00005BSTAB=00005D_Resolvent-Operator}
\end{equation}

\begin{thm}[{\cite[Thm.~3]{dafermos1973asymptotic}}]
\label{thm:=00005BIVP=00005D_Precompactness-Criterion} Let $\opA$
be the infinitesimal generator of a strongly continuous semigroup
of contractions on $\spaceState$. If $R(s,\opA)$ is compact for
some $s>0$, then $\gamma(\state_{0})$ is precompact for any $\state_{0}\in\spaceState$.
\end{thm}

Using Theorem~\ref{thm:=00005BIVP=00005D_Precompactness-Criterion}
reduces to proving that the embedding $\spaceDomain\subset\spaceState$
is compact, which based on the examples covered in this paper boils
down to proving that the embeddings
\begin{equation}
H_{\opDiv}(\Omega)\times H^{1}(\Omega)\subset\nabla H^{1}(\Omega)\times L^{2}(\Omega)\quad\left(\text{a}\right),\;L^{2}(\partial\Omega;\spaceDiff_{1})\subset L^{2}(\partial\Omega;\spaceDiff_{0})\quad\left(\text{b}\right)\label{eq:=00005BIVP=00005D_Compact-Embedding}
\end{equation}
are compact, where $\spaceDiff_{0}$ is the energy space of the extended
variables and $\spaceDiff_{1}\subset\spaceDiff_{0}$.

The compactness of the embedding (\ref{eq:=00005BIVP=00005D_Compact-Embedding}a)
is obvious if $d=1$. If $d=3$, it can be proven using the following
regularity result: if $\Omega$ is a bounded simply connected open
set with Lipschitz boundary, \cite[Thm.~2]{costabel1990regularity}
\[
H_{\text{curl}}(\Omega)\cap\left\{ \uac\in H_{\opDiv}(\Omega)\;\left|\;\uac\cdot\normal\in L^{2}(\partial\Omega)\right.\right\} \subset H^{\frac{1}{2}}(\Omega)^{d}
\]
and $\nabla H^{1}(\Omega)\subset H_{\text{curl}}(\Omega)$ \cite[Thm.~2.9]{giraultraviart1986femNavierStokes}.
(Note the stringent requirement that $\Omega$ be simply connected.)

The compactness of (\ref{eq:=00005BIVP=00005D_Compact-Embedding}b)
depends upon both $d$ and the impedance kernel $z$. If $d=1$, then
it holds true if $\spaceDiff_{1}\subset\spaceDiff_{0}$ is compact
(which is satisfied by the delay impedance covered in Section~\ref{sec:=00005BDELAY=00005D_Delay-impedance},
where $\spaceDiff_{1}=H^{1}\left(-\tau,0\right)$ and $\spaceDiff_{0}=L^{2}\left(-\tau,0\right)$,
but not by the diffusive impedances covered in Sections~\ref{sec:=00005BDIFF=00005D_Standard-diffusive-impedance}--\ref{sec:=00005BEXTDIFF=00005D_Extended-diffusive-impedance}
) or if both $\spaceDiff_{1}$ and $\spaceDiff_{0}$ are finite-dimensional
(which is verified for a rational impedance). If $d>1$, then it is
not obvious.

\providecommand{\href}[2]{#2}
\providecommand{\arxiv}[1]{\href{http://arxiv.org/abs/#1}{arXiv:#1}}
\providecommand{\url}[1]{\texttt{#1}}
\providecommand{\urlprefix}{URL }


\end{document}